\documentclass[11pt]{article}

\usepackage[margin=1in]{geometry}
\usepackage{amsmath,amsthm,amssymb,amsfonts,float}
\usepackage{listings}
\usepackage{pdfpages}

\makeatletter
\newcommand{\subjclass}[2][2010]{%
  \let\@oldtitle\@title%
  \gdef\@title{\@oldtitle\footnotetext{#1 \emph{Mathematics subject classification.} #2}}%
}
\makeatother

\usepackage[colorlinks=true, pdfstartview=FitV, linkcolor=blue,citecolor=blue, urlcolor=blue]{hyperref}

\usepackage[abbrev,lite,nobysame]{amsrefs}
\usepackage{color}
\usepackage{esint}

\usepackage{mathtools,enumitem,mathrsfs}

\usepackage[compact]{titlesec}

\usepackage{comment}

\mathtoolsset{showonlyrefs=true}
    
\newcommand{\ep}{\epsilon}
\newcommand{\eps}{\epsilon}

\newcommand{\leqc}{\lesssim}

\newcommand{\grad}{\nabla}

\newcommand{\strat}{\circ}

\newcommand{\norm}[1]{\left|\left| #1 \right|\right|}
\newcommand{\abs}[1]{\left| #1 \right|}
\newcommand{\set}[1]{\left\{ #1 \right\}}
\newcommand{\brak}[1]{\left\langle #1 \right\rangle} 

\newcommand{\D}{\mathbb{D}}
\newcommand{\R}{\mathbb{R}}

\newcommand{\N}{\mathbb{N}}

\newcommand{\C}{\mathbb{C}}

\newcommand{\Z}{\mathbb{Z}}
\newcommand{\T}{\mathbb{T}}
\newcommand{\K}{\mathbb{K}}
\renewcommand{\S}{\mathbb{S}}

\newcommand{\tensor}{\otimes}

\newcommand{\cM}{\mathcal{M}}

\newcommand{\dee}{\mathrm{d}}

\newcommand{\dt}{\dee t}

\newcommand{\dx}{\dee x}

\newcommand{\dy}{\dee y}

\newcommand{\dq}{\dee q}

\DeclareMathOperator{\Div}{\mathrm{div}}
\DeclareMathOperator{\Id}{\mathrm{Id}}
\DeclareMathOperator{\tr}{\mathrm{tr}}

\DeclareMathOperator{\Span}{\mathrm{span}}

\usepackage{bbm}
\newcommand{\1}{\mathbbm{1}}

\renewcommand{\P}{\mathbf{P}}

\newcommand{\EE}{\mathbf E}
\newcommand{\PP}{\mathbf P}

\newtheorem{theorem}{Theorem}[section]
\newtheorem{proposition}[theorem]{Proposition}
\newtheorem{corollary}[theorem]{Corollary}
\newtheorem{lemma}[theorem]{Lemma}
\newtheorem*{lemma*}{Lemma}

\newtheorem{assumption}{Assumption}

\theoremstyle{definition}
\newtheorem{definition}[theorem]{Definition}
\newtheorem{remark}[theorem]{Remark}

\numberwithin{equation}{section}
    
\begin{document}

\title{Chaos in stochastic 2d Galerkin-Navier-Stokes} 
% \subjclass{Primary: 37H15, 35H10. Secondary: 37D25, 58J65, 35B65}
\author{Jacob Bedrossian\thanks{\footnotesize Department of Mathematics, University of Maryland, College Park, MD 20742, USA \href{mailto:jacob@math.umd.edu}{\texttt{jacob@math.umd.edu}}. J.B. was supported by National Science Foundation CAREER grant DMS-1552826 and the Simons Foundation (2020 Simons Fellowship)} \and Sam Punshon-Smith\thanks{\footnotesize Division of Applied Mathematics,  Brown University, Providence, RI 02906, USA \href{mailto:punshs@brown.edu}{\texttt{punshs@brown.edu}}. This material was based upon work supported by the National Science Foundation under Award No. DMS-1803481.}}

\maketitle

\begin{abstract}
We prove that all Galerkin truncations of the 2d stochastic Navier-Stokes equations in vorticity form on any rectangular torus subjected to hypoelliptic, additive stochastic forcing are chaotic at sufficiently small viscosity, provided the frequency truncation satisfies $N\geq 392$. By ``chaotic'' we mean having a strictly positive Lyapunov exponent, i.e. almost-sure asymptotic exponential growth of the derivative with respect to generic initial conditions. 
A sufficient condition for such results was derived in previous joint work with Alex Blumenthal which reduces the question to the non-degeneracy of a matrix Lie algebra implying H\"ormander's condition for the Markov process lifted to the sphere bundle (projective hypoellipticity). 
The purpose of this work is to reformulate this condition to be more amenable for Galerkin truncations of PDEs and then to verify this condition using a) a reduction to genericity properties of a diagonal sub-algebra inspired by the root space decomposition of semi-simple Lie algebras and b) computational algebraic geometry executed by Maple\footnote{Maple is a trademark of Waterloo Maple Inc.} in exact rational arithmetic. 
Note that even though we use a computer assisted proof, the result is valid for all aspect ratios and all sufficiently high dimensional truncations; in fact, certain steps simplify in the formal infinite dimensional limit.  
\end{abstract}

\setcounter{tocdepth}{1}
{\small\tableofcontents}

%% Introduction
\section{Introduction}
%!TEX root = master.tex

Chaos is fundamental to our understanding of fluids and fluid-like systems in realistic settings and is thought to be an integral aspect of turbulence in these systems \cite{BMOV05}. However, there are few mathematically rigorous results on chaos in fluid models, even for finite dimensional models. In this paper we consider Galerkin truncations of the 2d Navier-Stokes equations on a torus of aspect ratio $r > 0$ and subjected to additive stochastic forcing. This system can be written as a member of the following class of stochastic differential equations (SDEs) on $\R^d$, 
\begin{align}
	\dee x_t = \left( B(x_t,x_t) - \eps Ax_t \right)\,\dt + \sqrt{\eps} \sum_{k=1}^r e_k \dee W^k_t \, , \label{eq:xtclass}
\end{align}
where $\{e_k\}_{k=1}^r$ are a family of constant vectors and $\{W^k\}_{k=1}^r$ are independent standard Wiener processes with respect to a canonical stochastic basis $(\Omega,\mathscr{F},(\mathscr{F}_t),\P)$. 
Here $A$ is symmetric positive definite, and $B$ is bilinear satisfying $x \cdot B(x,x) = 0$ and $\Div B = 0$ as well as the ``cancellation property'' $B(e_k,e_k) = 0$ (a more general cancellation property can be taken, see \cite{BBPS20}). 
Besides Galerkin truncations of the Navier-Stokes equations (see Section \ref{sec:GNSE} below), this class includes Lorenz-96 \cite{Lorenz1996},  and the shell models GOY \cite{Gledzer1973} and SABRA \cite{LvovEtAl98}, all of which are observed to be chaotic for $\eps$ small (see e.g. discussions in \cite{BBPS20,Majda16,KP10} for Lorenz-96, e.g. \cite{Ditlevsen2010} for GOY and SABRA, and e.g. \cite{BMOV05} for Galerkin-Navier-Stokes).   
The balance of dissipation and forcing (usually called `fluctuation-dissipation') is chosen so that there is a non-trivial limit as $\eps \to 0$; this balance can always be taken for small damping/large forcing regimes by a suitable re-scaling of $t$ and $x$ (see Remark \ref{rmk:rescale}). 

For many physical models, under fairly general conditions on $\set{e_k}_{k=1}^r$ it is possible to show that there is a unique stationary measure $\mu$ associated to the Markov process of $(x_t)$ solving \eqref{eq:xtclass}(see Section \ref{sec:Hypo} below). 
We denote the stochastic flow of diffeomorphisms defined by the solution map as $x \mapsto x_t =: \Phi^t_\omega(x), t \geq 0$.  
For SDE of the form \eqref{eq:xtclass}, the stationary measures have Gaussian upper bounds (see Section \ref{sec:Hypo} or \cite{BL20}), and so it is possible\footnote{This follows by the Kingman subadditive ergodic theorem; see e.g. \cite{kifer2012ergodic}.} to define a top Lyapunov exponent via the limit
\begin{align*}
\lambda_1 := \lim_{t \to \infty}\frac{1}{t} \log \abs{D_x \Phi^t_\omega},  
\end{align*}
which holds for $\mu \times \PP$ almost every $(x,\omega)\in \R^d\times \Omega$.
In particular, the Lyapunov exponent $\lambda_1$ is deterministic and well-defined independent of initial condition or random noise path. 
When $\lambda_1 > 0$, we say that \eqref{eq:general-SDE} \emph{chaotic} as it shows an exponential sensitivity of the trajectory to changes in the initial condition. 
For deterministic systems, verifying $\lambda_1 > 0$ is notoriously difficult; see e.g. the discussions in \cite{BBPS20} and \cite{young2013mathematical, pesin2010open, wilkinson2017lyapunov}. 
Even in the random case, there are relatively few methods. The methods \`a la Furstenberg (see e.g. \cite{carverhill1987furstenberg,virtser1980products, royer1980croissance, ledrappier1986positivity, baxendale1989lyapunov}) are powerful when applicable, but are not quantitative and cannot be used to obtain $\lambda_1 > 0 $ for dissipative systems. 
For systems with a lot of rigid structure, it is sometimes possible to obtain even asymptotic expansions of $\lambda_1$ in small noise limits; see e.g. \cite{ausMilstein, pardoux1988lyapunov, pinsky1988lyapunov, imkeller1999explicit, moshchuk1998moment,APW1986, baxendale2002lyapunov, baxendale2004vanderpol} however, these methods generally require almost complete knowledge of the limiting $\eps \to 0$ dynamics and it is far from clear how these arguments could be adapted to more complicated systems such as the Galerkin-Navier-Stokes equations or even Lorenz-96. 

Our recent work with Alex Blumenthal \cite{BBPS20} puts forward a new method for obtaining lower bounds on $\lambda_1$ for SDEs. 
Therein, we used the method to prove that the Lorenz-96 model subject to stochastic forcing is chaotic for all $\eps$ sufficiently small; the Lorenz-96 model is commonly used in applied mathematics as a test case for numerical or analytical methods for high-dimensional, chaotic systems \cite{Majda16,KP10,PazoEtAl08,BCFV02,OttEtAl04}, but no mathematical proof of chaos had previously been found even in the stochastic case. 
More generally, for each SDE in the class \eqref{eq:xtclass}, we formulated a sufficient condition for chaos in terms of a certain Lie algebra associated to the nonlinearity.  
In particular, the Lie algebraic condition of \cite{BBPS20} implies the quantitative estimate
\begin{align*}
\lim_{\eps \to 0} \frac{\lambda_1(\eps)}{\eps} = \infty, 
\end{align*}
and hence $\exists \eps_0 > 0$ such that  $\forall \eps \in (0,\eps_0)$ there holds $\lambda_1 > 0$.
 
In this paper, we first provide a convenient reformulation of the Lie algebra condition of \cite{BBPS20}, particularly amenable to application in Galerkin approximations of PDEs and other complex-valued SDEs, using basic concepts from complex geometry.
%After using additional reductions in complexity, 
Our main result is to verify this condition for the Galerkin truncations of the 2d Navier-Stokes equations with frequency cutoff $N \geq 392$ on torii of any aspect ratio (Theorem \ref{thm:mainchaos}), thus proving chaos for all $\eps$ sufficiently small.

Inspired by the classical root space decomposition of semi-simple Lie algebras, we reduce the problem to proving genericity of a diagonal sub-algebra.
Using the algebraic structure of the nonlinearity in Fourier space, we further reduce this question to showing that a certain list of polynomial systems have only trivial solutions.  
These are exhaustively verified to be inconsistent using methods from computational algebraic geometry carried out with Maple \cite{maple}. 
Note that despite using a computer assisted proof, our results nevertheless apply in arbitrary frequency truncation and arbitrary aspect ratio and, in a certain sense, well-suited for infinite dimensions.  
We believe the method put forward in this paper should be applicable to other Galerkin approximations of PDEs, both real and complex valued, provided the nonlinearity is a finite-degree polynomial. 

\subsection{2d Galerkin-Navier-Stokes equations} \label{sec:GNSE}
Denote the torus of arbitrary side-length ratio $\T^2_r = [0,2\pi) \times [0, \frac{2\pi}{r})$ (periodized) for $r > 0$.  
Recall that the Navier-Stokes equations on $\T^2_r$ in vorticity form are given by 
\begin{align*}
\partial_t w + u \cdot \grad w = \eps \Delta \omega + \sqrt{\eps} \dot{W}_t,  
\end{align*}
where $u$ is the divergence free velocity field satisfying the Bio-Savart law $u = \nabla^{\perp}(-\Delta)^{-1}w$ and $\dot{W}_t$ is a white-in time, colored-in-space Gaussian forcing assumed to be diagonalizable with respect to the Fourier basis. 
The parameter $\eps$ represents the kinematic viscosity; the noise has been scaled with a matching $\sqrt{\eps}$ so that the dynamics have a non-trivial limit when $\eps \to 0$. 
For definiteness, we will assume the forcing is of the form
\begin{align*}
W_t = 2\sum_{k \in \mathbb Z^2_+} \alpha_k (\cos (k_1 x_1 + r k_2 x_2)) W_t^{(k;a)} + \beta_k (\sin (k_1 x_1 + r k_2 x_2)) W_t^{(k;b)}, 
\end{align*}
 $\{W^{(k,a)}, W^{(k,b)}\}_{k=1}^r$ are independent standard Wiener processes with respect to a canonical stochastic basis $(\Omega,\mathscr{F},(\mathscr{F}_t),\P)$ 
and that if $\alpha_k \neq 0$ then $\beta_k \neq 0$ and vice-versa. 
Here we are denoting the ``upper'' lattice
\[
\Z^2_+ := \set{(k_1,k_2) \in \Z^2_0: k_2 > 0}\cup \set{(k_1,0)\in \Z^2_0 \,:\, k_1 > 0},
\]
where $\Z^2_0 := \Z^2\backslash\{0\}$. We denote the set of driving modes by
\[
	\mathcal{K} := \{k\in \Z^2_+ \,:\, \alpha_k, \beta_k \neq 0\}.
\]

Upon taking the Fourier transform one can re-write the equations in terms of the complex coefficient $w_k = \frac{r}{(2\pi)^2}\int_{\T^2_r} e^{- i \left( x_1 k_1 + r x_2 k_2\right)} w(x)\,\mathrm{d}x$ for each $k\in \Z^2_0$ and satisfies the reality constraint $w_{-k} = \overline{w_k}$. In Fourier space, the nonlinearity $B(w,w) = -u\cdot\nabla w$ takes the form for each $\ell\in \Z^2_0$
\begin{equation}\label{eq:B-def}
	B_\ell(w,w) := \frac{1}{|\T_r^2|}\int_{\T^2_r} B(w,w) e^{-i (\ell_1 x_1 + r \ell_2 x_2)}\dx = \frac{1}{2}\sum_{k+j=\ell} c_{j,k} w_j w_k,
\end{equation}
where the sum is over all $j,k \in Z^2_0$ such that $j+k=\ell$, the symmetrized coefficient is
\[
    c_{j,k} := \langle j^\perp,k\rangle_r\left(\frac{1}{|k|^2_r} - \frac{1}{|j|^2_r}\right), 
\]
and we are using the notation
\begin{align*}
\langle j^\perp,k\rangle_r = r (j_2k_1 - j_1k_2), \quad \abs{k}_r^2 = k_1^2 + r^2 k_2^2. 
\end{align*}
In what follows $c_{j,k}$ always depends on $r$ but we suppress the dependence for notational simplicity. 
One way to deal with the reality constraint $w_{-k} = \overline{w}_k$ is to restrict the complex valued $w_k$ to the upper lattice $\Z^2_+$ and encode the values in the negative lattice $\mathbb{Z}_{-}^2:= - \Z^2_+$ through complex conjugation $w_{-k} = \overline{w}_k$. In this sense we can think of the vorticity $w = (w_\ell)$ as belonging to the complex space $\C^{\Z^2_+}$, and the Navier-Stokes equations is seen to be the following complex-valued evolution equation on $\C^{\Z^2_+}$
\begin{align}
	\dot{w}_\ell = B_\ell(w,w) - \nu \abs{\ell}^2_r w_\ell + \sqrt{\nu} \left(  \alpha_\ell \dot{W}_t^{(\ell;a)} + i \beta_\ell \dot{W}_t^{(\ell;b)} \right). \label{def:NSEFF}
\end{align}
The above formulation gives a clear method for finite-dimensional approximation, known as a Galerkin approximation. 
Define the truncated lattice 
\[
\Z^2_{+,N} = \set{k \in \Z^2_+ : \abs{k}_{\ell^\infty} \leq N}, \quad |k|_{\ell^\infty}  := \max\{|k_1|,|k_2|\},
\]
and now simply restrict the vorticity to the truncated lattice $w = (w_\ell) \in \C^{\Z^2_{+,N}}$, in which case \eqref{def:NSEFF} becomes an SDE, with the sum in the non-linearity \eqref{eq:B-def} now taken over all $j,k \in \Z^2_{+,N}$ such that $j+k=\ell\in \Z^2_{+,N}$. We regard the phase space as a real finite-dimensional manifold $\C^{\mathbb Z^2_{+,N}} \cong (\R^2)^{\mathbb Z^2_{+,N}}$
with the real and imaginary coordinates $\{(a_k,b_k)\}_{k\in \Z^2_{+,N}}$ defined by $w_k = a_k + i b_k$ giving also the corresponding basis $\set{\partial_{a_k},\partial_{b_k}}_{k \in \Z^2_{+,N}}$ for the tangent space $T_w\C^{\mathbb Z^2_{+,N}}$. 
One can easily check that this truncation satisfies all of the hypotheses assumed for \eqref{eq:xtclass}. 

\subsection{Main results}

We will assume a general condition on $\mathcal{K}$ that implies there exists a unique stationary measure for all $\eps> 0$ (c.f. \cite{E2001-lg,HM06,Romito2004-rc}). 
Denote the full truncated lattice by $\Z^2_{0,N} := \set{k \in \Z^2_0 : \abs{k}_{\ell^\infty} \leq N}$.  
\begin{assumption} \label{ass:hypo}
Define the sets $\mathcal{Z}^n \subset \Z^2_{0,N}$, 
\begin{align*}
\mathcal{Z}^0 & = \mathcal{K} \cup (-\mathcal{K}) \\ 	
\mathcal{Z}^n & = \set{k \in \Z^2_{0,N} :  k = k_1 + k', \; k_1 \in \mathcal{Z}^{n-1}, \; k' \in \mathcal{Z}^0, \; c_{k_1,k'} \neq 0 }. 
\end{align*}
We say $\mathcal{K}$ is \emph{hypoelliptic} if $\Z_{0,N}^2 = \bigcup \mathcal{Z}^n$. 
\end{assumption}
In \cite{E2001-lg} it was shown explicitly that for $r=1$ the sets $\mathcal{K} = \set{(0,1),(1,1)}$ and $\mathcal{K} = \set{(1,0),(1,1)}$ are both hypoelliptic for all $N$ (note also that if $\mathcal{K}$ is hypoelliptic, then so is any $\mathcal{K}'$ such that $\mathcal{K} \subseteq \mathcal{K}'$ ). 
In the limit $N \to \infty$ the set of hypoelliptic forcings is easier to characterize \cite{HM06}, however for a fixed $N$ we are unaware of a simple characterization of all hypoelliptic $\mathcal{K}$ due to the presence of the truncation.

The main theorem of this work is the following. 
\begin{theorem} \label{thm:mainchaos}
Consider the 2d Galerkin-Navier-Stokes equations with frequency truncation $N \geq 392$ on $\T^2_r$. 
Suppose that $\mathcal{K}$ is hypoelliptic. Then
\begin{align}
	\lim_{\eps \to 0} \frac{\lambda_1(\eps)}{\eps} = \infty, \label{eq:chaos}
\end{align}
and in particular $\forall N \geq 392$ and $\forall r > 0$,  $\exists \eps_0 > 0$ such that for all $\eps \in (0,\eps_0)$, $\lambda_1 > 0$.
\end{theorem}

Before we make remarks, let us provide an outline of the remainder of the paper.
In Section \ref{sec:Hypo} we recall the definition of \emph{projective hypoellipticity} which corresponds to H\"ormander's condition for the Markov process $(x_t)$ lifted 
to the sphere bundle in a suitable manner, and we recall our results with Alex Blumenthal \cite{BBPS20} which (A) provide a useful sufficient condition for projective hypoellipticity in terms of a matrix Lie algebra based only on the nonlinearity (Proposition \ref{prop:class-proj-span}) and (B) show that projective hypoellipticity implies Theorem \ref{thm:mainchaos} (see Section \ref{sec:GNSEproSpan}). 
In Section \ref{sec:ProjComp}, we reformulate the sufficient condition for projective hypoellipticity to be more suitable for \eqref{def:NSEFF} (Proposition \ref{prop:sufficient-projectspan-NS}).  
The remainder of the paper is dedicated to proving this sufficient condition. 
Section \ref{sec:diagonal} introduces a diagonal sub-algebra $\mathfrak{h}$ and shows that a certain genericity property of $\mathfrak{h}$ implies projective hypoellipticity (Corollary \ref{cor:disDon}).
Section \ref{sec:distinct} proves this genericity property (Proposition \ref{prop:main-distinctness-prop}). 
Section \ref{sec:distinct} also contains a more detailed summary of the proof of Theorem \ref{thm:mainchaos}  which puts together all of the pieces.  
Sections \ref{sec:diagonal} and \ref{sec:distinct} both use computational algebraic geometry and computer assisted proofs performed with Maple \cite{maple} to compute Gr\"obner bases for certain polynomial ideals  (although the arguments used in Section \ref{sec:distinct} are significantly more complicated).
A review of the algebraic geometry required is included for the readers' convenience in Appendix \ref{sec:Groebner} and the computer code is included in \ref{sec:code}. 
Appendix \ref{sec:lattice} contains a simple but crucial technical lemma regarding polynomial ideals.

\subsection{Remarks}

\begin{remark}
We did not attempt to optimize the proof to try and reduce the value of $N$ and we believe that the result holds for much smaller $N$ as well. 
However, $N \geq 392$ is already enough to treat nearly all modern numerical simulations of the 2d Navier-Stokes equations on $\mathbb T^2_r$.  
\end{remark}

\begin{remark}
As might be expected, we currently do not have any quantitative estimates on $\eps_0$ and $\lambda_1$ in terms of $N$ at this time. 
\end{remark}

\begin{remark}
The quantitative estimate of $\lambda_1$ in terms of $\eps$ is almost certainly sub-optimal 
\end{remark}

\begin{remark} \label{rmk:rescale}
If one starts with the scaling 
\begin{align}
	\dee x_t = \left( B(x_t,x_t) - \hat{\eps} Ax_t \right)\,\dt + \sum_{k=1}^r e_k\, \dee W^k_t \, , \label{eq:xtclasslame}
\end{align}
we can relate the stochastic flow of diffeomorphisms $\hat \Phi^t_{\hat\omega}$ solving the SDE \eqref{eq:xtclasslame}
with the stochastic flow $\Phi^t_\omega$ solving \eqref{eq:xtclass} by 
$\Phi^t_\omega(u) = \sqrt{\hat{\epsilon}} \widehat{\Phi}^{\sqrt{\hat{\epsilon}} t}_{\hat{\omega}} (u/\sqrt{\hat\ep})$ (where $\omega_t = \hat\ep^{-1/4}\hat{\omega}_{\sqrt{\hat\ep}t}$ a Brownian self-similar rescaling of the noise path $\hat\omega$ so equality of the two flows is interpreted as \emph{equality in probabilistic law}).  
Thus, the Lyapunov exponent $\hat \lambda_1^{\hat \epsilon}$ of the stochastic flow $\hat \Phi^t_\omega$ satisfies $\hat \epsilon^{-1} \hat \lambda_1^{\hat \epsilon} = \epsilon^{-1} \lambda_1^\epsilon$, and in particular $\hat{\lambda}^{\hat \epsilon }_1 > 0$ if and only if $\lambda_1^{\epsilon } > 0$. 
\end{remark}

\begin{remark}
We believe our methods should extend in an analogous way to Galerkin truncations of other PDE with polynomial nonlinearities, for example the 3D Navier-Stokes equations, as well as more general truncations like the Fourier decimation models in e.g. \cite{FrischEtAl2012}.
\end{remark}

\begin{remark}
Another important truncation of the Euler non-linearity is the Zeitlin model (see \cite{Zeitlin1991-qf,Zeitlin1991-gk,Gallagher2002-sp}), which has the added benefit that it preserves the Poisson structure of the Euler equations for the co-adjoint orbits on $\text{SDiff}(\T^2)$ (see \cite{Arnold1998-hz}). In this approximation, instead of a sharp truncation in frequency, the Fourier modes are taken to belong to the periodic lattice $\Z^2_{0,\text{mod }N} := \Z^2_0\backslash N \Z^2_0$ for some $N\geq 1$ and the non-linearity is given by
\[
	B_\ell(w,w) = \frac{1}{2}\sum_{\ell= j + k \text{ mod }N} \sin\left(\frac{2\pi \langle j^\perp,k\rangle}{N}\right)\left(\frac{1}{|k|^2} - \frac{1}{|j|^2}\right)w_jw_k.
\]
While we expect that our results should still hold for this model, it is important to note that our methods do not currently apply to this truncation since the multiplier $\sin(2\pi \langle j^\perp,k\rangle/N )$ is not a polynomial in the lattice variables $j,k$ and therefore cannot be easily treated by our algebraic geometry methods.
\end{remark}

\begin{remark}
The use of computational algebraic geometry methods to deduce generating properties of matrix Lie algebras is not an entirely new idea. For instance, computations using polynomial ideals and Gr\"obner bases feature in \cite{Elliott2009} as a practical tool for deducing transitivity on $\R^n$ for certain matrix algebras related to bilinear control systems. However, in contrast with our work, the techniques used in \cite{Elliott2009} depend very strongly on the dimension and do not generalize to infinite or arbitrary dimensional systems like ours.
\end{remark}

\begin{remark}
Our computer assisted proof uses Maple's implementation of the F4 algorithm \cite{Faugere1999-sj} to compute the reduced Gr\"obner basis (see \cite{Cox} or Appendix \ref{sec:Groebner}) of certain polynomial ideals associated with the coefficient $c_{j,k}$. Computing Gr\"obner bases, particularly for ideals generated by high degree polynomials with many variables, can be notoriously costly and can be very sensitive to the choice of variable ordering and associated monomial ordering. It is important to remark that the set up of several of the computations included in Appendix \ref{sec:code} are incredibly delicate, and often fail to converge if some of the constraints aren't included, the variable ordering isn't chosen correctly, or if the choice of saturating polynomial isn't written in a certain way.
Indeed, the principle part of the calculation takes places on polynomials of degree 19 in 11 independent variables, which is a far too high dimensional space in which to do arbitrary computations, even for modern supercomputers. 
\end{remark}

\begin{remark}
A remarkable feature of our proof is that it holds for {\em all} $N$ large enough and for {\em all} torus aspect ratios $r>0$. Such a conclusion is simply not possible using more direct methods. Specifically, for a given fixed $N$ and fixed $r > 0$, one can of course attempt to check the matrix algebra generating properties exhaustively using a more direct method, however computing this very quickly becomes extremely expensive for even fairly modest $N \geq 10$ and can only be done using exact rational arithmetic if $r$ is a rational number.
\end{remark}

\subsection{Acknowledgements}
We would like to give a special mention to Alex Blumenthal for many fruitful discussions and whose work in \cite{BBPS20} laid the foundation for this one.

\section{Preliminaries: projective hypoellipticity and chaos} \label{sec:Hypo}

In this section we review the concepts of hypoellipticity for Markov semigroups, the projective process and its hypoellipticity, and the main results of \cite{BBPS20} connecting projective hypoellipticity to chaos as well as some convenient characterizations of projective hypoellipticity.

\subsection{Hypoellipticity} \label{sec:generalCondProjSpan}
In this section we briefly recall the notion of hypoellipticity and its relevance to SDEs. 

In what follows, let $(M,g)$ be a smooth, connected, complete Riemannian manifold without boundary and $\mathfrak{X}(M)$, the space of smooth vector fields over $M$. 
Let $[X,Y]$ be the Lie bracket of two vector fields $X,Y$,
defined for each $f\in C^\infty(M)$ by
\[
	[X,Y](f) = XY(f) - YX(f)
\]
where $X(f)$ denotes the directional derivative of $f$ in the direction $X$. This bracket turns $\mathfrak{X}(M)$ into an infinite dimensional Lie algebra. 
Denote the adjoint action $\mathrm{ad}(X):\mathfrak{X}(M)\to \mathfrak{X}(M)$ by $\mathrm{ad}(X)Y = [X,Y]$. The next condition, introduced in \cite{H67}, is crucial to our study. 
\begin{definition}[H\"ormander's condition] \label{def:Hormander}
For a given collection $\mathcal{F} \subseteq \mathfrak{X}(M)$ define the {\em Lie algebra generated by $\mathcal{F}$} by
\begin{equation}\label{eq:lie-def}
	\mathrm{Lie}(\mathcal{F}) := \Span\{\mathrm{Lie}^m(\mathcal{F})\,:\, m\geq 1\},
\end{equation}
where 
\[
	\mathrm{Lie}^m(\mathcal{F}) := \Span\{\mathrm{ad}(X_r)\ldots\mathrm{ad}(X_{2})X_1\,:\, X_i \in \mathcal{F}\,,\, 1\leq r\leq m\}.
\]
We say that a collection of smooth vector fields $\mathcal{F}\subseteq \mathfrak{X}(M)$ satisfies {\em H\"{o}rmander's condition} on $M$ if for each $x\in M$ we have the following spanning property
\[
	\mathrm{Lie}_x(\mathcal{F}):=\{ X(x)\,:\, X\in \mathrm{Lie}(\mathcal{F})\} = T_xM.
\]
\end{definition}

It is also useful to define a notion of (locally) uniform spanning properties of a collection of vector fields.

\begin{definition}[Uniform H\"ormander] \label{def:UniHormander}
	Let $\mathcal{F}^\eps \subset \mathfrak{X}(M)$ be a set of vector fields parameterized by $\epsilon \in (0,1]$. We say $\mathcal{F}^\ep$ satisfies the uniform H\"ormander condition on $M$ if $\exists m \in \mathbb N$, such that for any open, bounded set $U \subseteq \cM$ there exists constants $\set{K_n}_{n=0}^\infty$, such that for all $\epsilon \in (0,1]$ and all $x \in U$, there is a finite subset $\mathcal{V}_x \subset \mathrm{Lie}^m(\mathcal{F}^\ep)$ such that $\forall \xi \in \R^d$
	\begin{align*}
		\abs{\xi} \leq K_0 \sum_{X \in \mathcal{V}_x} \abs{X(x) \cdot \xi} \qquad \sum_{X \in\mathcal{V}_x}\norm{X}_{C^n(U)} \leq K_n. 
	\end{align*}
\end{definition}

An important role will also be played by a certain Lie algebra ideal which is better suited to hypoellipticity for parabolic equations and Markov semigroups.
\begin{definition}[Parabolic H\"ormander's condition] \label{def:PHormander}
Let $X_0\in \mathfrak{X}(M)$ be a distinguished ``drift'' vector field and let $\mathcal{X} \subseteq \mathfrak{X}(M)$ be of a collection of ``noise'' vector fields. We define the {\em zero-time ideal} generated by $X_0$ and $\mathcal{X}$ as the Lie algebra generated by the sets $\mathcal{X}$ and $[\mathcal{X},X_0] := \{[X,X_0]\,:\, X\in \mathcal{X}\}$, which we denote by
\[
	\mathrm{Lie}(X_0;\mathcal{X}) := \mathrm{Lie}(\mathcal{X},[\mathcal{X},X_0]).
\]
Correspondingly we say that the vector fields $X_0,\mathcal{X}$ satisfy the {\em parabolic H\"ormander condition} on $M$ if the vector fields $\mathrm{Lie}_x(X_0;\mathcal{X}) = T_xM$. Likewise we say $X_0$,$\mathcal{X}$ satisfies {\em the uniform parabolic H\"ormander condition} if $\{\mathcal{X},[\mathcal{X},X_0]\}$ satisfies the uniform H\"ormander condition as in Definition \ref{def:UniHormander}. 
\end{definition}

\begin{remark}
The terminology `zero-time ideal' comes from geometric control theory (see e.g. \cite{JurdGCT}) where $\mathrm{Lie}(X_0;\mathcal{X})$ plays an important role in obtaining exact controllability of affine control systems. A proof that the definition of the zero-time ideal in geometric control theory and $\mathrm{Lie}(X_0;\mathcal{X})$ coincide can be found in [Proposition 5.10, \cite{BBPS20}], although this fact is likely well-known to experts in geometric control theory (see e.g. discussion in \cite{Elliott2009} Chapter 3). 
\end{remark}

Consider a stochastic process $x_t \in M, t \geq 0$ defined by the (Stratonovich) SDE 
\begin{equation}\label{eq:general-SDE}
	\dee x_t = X_0(x_t)\,\dt + \sum_{k=1}^r X_k(x_t)\strat\dee W^k_t \, , 
\end{equation}
for vector fields $X_k \in \mathfrak{X}(M)$. 
Define the Markov kernel for any set $O \subset M$ and $x\in M$, $P_t(x,O) = \PP(x_t \in O \,|\, x_0 = x)$ and define the Markov semigroup on 
\begin{align*}
\mathcal{P}_t \varphi(x) := \EE\left(\varphi(x_t) \,|\, x_0 = x \right) = \int_M \varphi(y) P_t(x, \dy)
\end{align*}
where $\varphi : M \to \R$ is bounded and measurable. We also define the adjoint semigroup on probability measures $\mathcal{P}(M)$ for each Borel $A\subset M$ and $\mu\in \mathcal{P}(M)$
\[
	\mathcal{P}_t^*\mu(A) := \int_M P_t(y,A)\mu(\dy).
\]
Under fairly mild conditions on the vector fields $\set{X_k}_{k=0}^r$, these Markov semigroups are well defined and solve deterministic PDEs \cite{kunita1997stochastic, arnold1995random}. 
Recall the definition of stationary measure for an SDE. 
\begin{definition}
A measure $\mu \in \mathcal{P}(M)$ is called \emph{stationary} for a given SDE if $\mathcal{P}_t^\ast \mu = \mu$. 
\end{definition}
H\"ormander's theorem implies that if $X_0,\set{X_1,...,X_r}$ satisfies the parabolic H\"ormander condition, then $\mathcal{P}_t : L^\infty \to C^\infty$ (see e.g. \cite{H67,H11}). 
This implies that any stationary measure $\mu$ is absolutely continuous with respect to Lebesgue measure with a smooth density. 
By the Doob-Khasminskii theorem, this together with topological irreducibility\footnote{It suffices to show that for $t > 0$, $x \in M$, $O \subset M$ open, then $P_t(x,O) > 0$, however if only uniqueness is desired, one can get by with much less.} implies the uniqueness of stationary measures.  

\subsection{Projective Hypoellipticity}
It is well-known that many dynamical properties of the general SDE \eqref{eq:general-SDE} are encoded in the process
\[
 z_t = (x_t,v_t) := \left(\Phi^t(x),\frac{D_x\Phi^t v }{\abs{D_x\Phi^t v }}\right). 
\]
The process $(z_t)$ takes values on the unit tangent bundle $\S M$ defined by the fibers $\S_x M = \S^{n-1} (T_x M)$ and is called the \emph{projective process} (as one can just as well consider the process on the projective bundle $PM$).  
One can show that $z_t$ solves the lifted version of \eqref{eq:general-SDE} on $\S M$
\begin{align*}
	\dee z_t = \widetilde{X}_0(z_t)\,\dt + \sum_{k=1}^r \widetilde{X}_k(z_t) \circ \dee W_t^{k}. 
\end{align*}
Here, for a smooth vector field $X$ on $M$, define the ``lifted'' vector field $\widetilde{X}$ on $\S M$ by
\[
	\widetilde{X}(x,v)  := (X(x) , V_{\nabla X}(x,v)),
\]
where each of the components in the block vector above is determined via the orthogonal splitting $T_{(x,v)} \S M = T_xM \oplus T_{v} \S_x M$ into horizontal and vertical components induced by the Levi-Civita connection $\nabla$ on $M$ and the associated Sasaki metric\footnote{The Sasaki metric (see \cite{Sasaki1962-rt}) is the unique metric on $\S M$ induced from $g$ such that the splitting $T_{(x,v)} \S M = T_xM \oplus T_{v} \S_x M$ induced by the Levi-Civita connection is orthogonal.} $\tilde{g}$ on $\S M$ . The ``vertical'' component $V_{\nabla X}$ will be referred to as the {\em projective vector field} and is defined explicitly by
\[
V_{\nabla X}(x,v) := \nabla X(x)v - \langle v,\nabla X(x)v\rangle_{x}v,
\]
where $\nabla X(x)$ denotes the total covariant derivative of $X$, viewed as a linear endomorphism on $T_xM$. 
That there should be a connection between the hypoellipticity of the projective process and the Lyapunov exponents is well-documented (see e.g. \cite{pinsky1988lyapunov,baxendale1989lyapunov,dolgopyat2004sample}). 
Indeed, the sufficient condition proved in \cite{BBPS20} for \eqref{eq:chaos} in systems of the form \eqref{eq:xtclass} is the requirement of uniform hypoellipticity of the $(z_t)$ process, i.e. \emph{projective hypoellipticity}, which we explain next. 

Here we recall necessary and sufficient conditions on a collection of vector fields $\mathcal{F}\subseteq \mathfrak{X}(M)$ so that their lifts $\widetilde{\mathcal{F}} = \{\widetilde{X}\,:\, X\in \mathcal{F}\} \subseteq \mathfrak{X}(\S M)$ satisfy the H\"ormander condition on $\S M$. Since the vector fields $\mathcal{F}$ may not be volume preserving, it is convenient to define for each $X\in \mathfrak{X}(M)$ and $x\in M$ the following traceless linear operator on $T_xM$ : 
\[
	M_X(x) := \nabla X(x) - \tfrac{1}{n}\Div X(x) \Id \, , 
\]
which we view as an element of the Lie algebra $\mathfrak{sl}(T_xM)$ of linear endomorphisms $A$ with $\tr(A) =0$ and Lie bracket given by the commutator $[A,B] = AB-BA$. Since the projective vector field $V_{\nabla X}(v)$ includes a projection orthogonal to $v$, we always have $V_{\nabla X} = V_{M_X}$. For each $x\in M$, an important role will be played by the following Lie sub-algebra of $\mathfrak{sl}(T_xM)$ 
\begin{equation}\label{eq:m-lie-alg-def}
	\mathfrak{m}_x(\mathcal{F}) := \{M_X(x) \,:\, X\in \mathrm{Lie}(\mathcal{F})\,,\, X(x) =0\}.
\end{equation} 
Note that $\mathfrak{m}_x(\mathcal{F})$ is independent of any choice of coordinates (and is in fact independent of the choice of metric). One can further check that $\mathfrak{m}_x(\mathcal{F})$ is indeed a Lie sub-algebra of $\mathfrak{sl}(T_xM)$.

The spanning properties of the lifted vector fields $\widetilde{\mathcal{F}}$ on $\S M$ can be related to properties of the Lie algebra $\mathfrak{m}_x(\mathcal{F})$. An important role is played by the non-trivial fact that the lifting map $X\mapsto \widetilde{X}$ satisfies the identity
\[
	[\widetilde{X},\widetilde{Y}] = [X,Y]{\,\,}^{\widetilde{}},
\] 
and therefore is a Lie algebra isomorphism\footnote{This was observed in \cite{baxendale1989lyapunov}, but see e.g.  [Lemma C.2 \cite{BBPS20}] for a complete proof.} onto the set of lifts 
\[
\widetilde{\mathfrak{X}}(M) := \{\widetilde{X}\,:\, X\in \mathfrak{X}(M)\}.
\]
The associated implications for projective hypoellipticity of the lifts are conveniently recorded in the following from \cite{BBPS20}. 
\begin{proposition}[Proposition 2.7 in \cite{BBPS20}]\label{prop:class-proj-span}
Let $\mathcal{F}\subseteq \mathfrak{X}(M)$ be a collection of smooth vector fields on $M$. Their lifts $\widetilde{\mathcal{F}}\subseteq \widetilde{\mathfrak{X}}(M)$ satisfy the H\"ormander condition on $\S M$ if and only if $\mathcal{F}$ satisfies the H\"ormander condition on $M$ and for each $x\in M$, $\mathfrak{m}_x(\mathcal{F})$ acts transitively on $\S_xM$ in the sense that for each $(x,v)\in \S M$, one has
\[
\{V_A(x) \,:\, A \in \mathfrak{m}_x(\mathcal{F})\} = T_v\S_xM.
\]
In particular this implies that $\widetilde{\mathcal{F}}\subseteq \widetilde{\mathfrak{X}}(M)$ satisfies H\"ormander's condition if for each $x\in M$
\[
\mathrm{Lie}_x(\mathcal{F})= T_xM,  \quad\text{and}\quad \mathfrak{m}_x(\mathcal{F}) = \mathfrak{sl}(T_x M).
\]
\end{proposition}

\begin{remark}\label{rem:so-remark}
In general one should expect that ``generically'' $\mathfrak{m}_x(\mathcal{F}) = \mathfrak{sl}(T_xM)$ holds true. Indeed, it well-known in the control theory literature (see \cite{Boothby1979-rs}) that there is an open and dense set of $\mathfrak{sl}_n(\R)$ such that any two matrices $A,B$ in that set generate $\mathfrak{sl}_n(\R)$.
\end{remark}

\subsection{Chaos and Fisher Information} \label{sec:FI}

In this section we briefly recall some of the main results of \cite{BBPS20} for the readers' convenience. 
For this we have to define the \emph{sum Lyapunov exponent}, which describes the asymptotic exponential rate of the volume compression/expansion: 
\begin{align*}
\lambda_{\Sigma} = \lim_{t \to \infty} \frac{1}{t} \log\,\mathrm{det} (D_x\Phi^t_\omega).
\end{align*}
With some additional mild integrability (see \cite{kifer1988note,BBPS20} for discussions) the Kingman subadditive ergodic theorem \cite{kingman1973subadditive,kifer2012ergodic} implies that a unique stationary measure leads to uniquely defined $\lambda_1,\lambda_\Sigma$ attained for $\mu \times \PP$ a.e. $(x,\omega)$. 

For general SDE of the form \eqref{eq:general-SDE}, with Alex Blumenthal, we provided the following identity connecting a degenerate Fisher information-type quantity with the Lyapunov exponent.
\begin{proposition}[Proposition 3.2, \cite{BBPS20}] \label{prop:FI}
Assume that the SDE \eqref{eq:general-SDE} defines a global-in-time stochastic flow of $C^1$ diffeomorphisms and that the associated projective process $(z_t)$ has a unique stationary measure $\nu$ which is absolutely continuous with respect to the volume measure $\dq$ on $\S M$ with smooth density $f$ and which satisfies some additional mild decay and integrability estimates (see \cite{BBPS20} for details). 
Then,  
\begin{align*}
FI(f) := \sum_{k=1}^r \int_{\S M} \frac{ \abs{X_k^\ast f}^2 }{f} \,\dee q = n \lambda_1 - 2 \lambda_{\Sigma},  
\end{align*}
where $n$ is the dimension of $M$ and $\dee q$ the Riemannian volume measure on $\S M$, and $X_k^*$ denotes the formal adjoint of $X_k$ as a differential operator with respect to $L^2(\dee q)$.
\end{proposition}
\begin{remark}
A sharper version of the identity holds on the conditional measures with $n\lambda_1 - \lambda_\Sigma$ on the right-hand side, providing a  time-infinitesimal analogue of relative entropy inequalities studied in e.g. \cite{furstenberg1963noncommuting, baxendale1989lyapunov,ledrappier1986positivity}; see \cite{BBPS20} for details. 
\end{remark}
In \cite{BBPS20} we proved the following crucial uniform H\"ormander-type lower bound on the Fisher information, connecting regularity in $W^{s,1}$ of $f^\eps$ to the Fisher information and therefore the Lyapunov exponents.   
\begin{theorem} [Theorem 4.2, \cite{BBPS20}] \label{thm:FIlow}
Consider the SDE \eqref{eq:general-SDE} for vector fields $\set{X_0^\eps,...,X_r^\eps}$ parameterized by $\eps \in (0,1]$. 
Suppose that $\{\widetilde{X}_0^\eps,\widetilde{X}_1^\eps,...,\widetilde{X}_r^\eps\}$ satisfies the uniform H\"ormander condition on $\S M$ and suppose that for all $\eps \in (0,1]$ there exists a unique stationary measure $\nu$ with smooth density $f^\eps$ for the associated projective process $(z_t)$.  
Then $\exists s_\star \in (0,1)$ such that $\forall U \subset \S M$ open geodesic ball, $\exists C_U > 0$ such that $\forall \eps \in (0,1]$ there holds 
\begin{align*}
\norm{\chi_{U}f^\eps}_{W^{s_\star,1}}^2 \leq C_U \left(1 + FI(f^\eps) \right),
\end{align*}
where $\chi_U$ is a smooth cutoff function equal to 1 inside $U$ and outside of a slightly larger ball $U^\prime$.
Note both $s_\star$ and $C_U$ are independent of $\eps$. 
\end{theorem}
The above two results give a clear path towards estimating Lyapunov exponents from below if lower bounds on the regularity of $f^\eps$ can be obtained. 

\subsection{Application to the Galerkin-Navier-Stokes equations} \label{sec:GNSEproSpan}
In the context of the SDE in the specific class \eqref{eq:xtclass}, one can prove the following using standard methods.  
As discussed in Section \ref{sec:GNSE}, the Galerkin-Navier-Stokes equations written in Fourier variables and when phase space is interpreted through the real and imaginary parts as $\R^d$, the following theorem applies.  
\begin{theorem}[See \cite{BBPS20} or \cite{BL20}] 
	Let $X_0^\eps(x) = B(x,x) - \eps Ax$ and consider the class of SDE \eqref{eq:xtclass}. 
	These SDEs each generate families of global-in-time, smooth stochastic diffeomorphisms $\Phi_{\omega}^t$, and if $\set{X_0^\eps, X_1,...X_r}$ satisfies the parabolic H\"ormander condition, then for all $\eps > 0$, there exists a unique stationary measure $\mu$ with a smooth, density $\rho$ which satisfies a pointwise Gaussian upper bound.  
	Moreover, there exists a top Lyapunov exponent $\lambda_1(\eps) \in \R$ and a sum Lyapunov exponent $\lambda_\Sigma(\eps)$ such that the following limit holds $\mu \times \PP$ almost-surely
	\begin{align*}
		\lambda_1 & := \lim_{t \to \infty}\frac{1}{t} \log \abs{D_x \Phi^t_\omega}\\ 
		\lambda_\Sigma & := \lim_{t \to \infty}\frac{1}{t} \log \,\mathrm{det} (D_x \Phi^t_\omega). 
	\end{align*}
\end{theorem}

\begin{remark}
	In fact, if $X_0^\eps, \set{X_1,...X_r}$ satisfies the uniform parabolic H\"ormander's condition, then one can prove the pointwise Gaussian upper bound on $\rho$ \emph{uniformly} in $\eps$, as well as a uniform-in-$\eps$ strictly positive lower bound on all compact sets \cite{BL20}. 
\end{remark}

The main theorem of this paper is a description of the Lie algebra $\mathfrak{m}_x(T_xM)$ for the 2d Galerkin-Navier-Stokes equations \eqref{def:NSEFF}. 
\begin{theorem}\label{thm:Main-projective-spanning}
	Let $N \geq 392$, let $X_0(w) = B(w,w) + \ep \Delta w$ be the Galerkin Navier-Stokes vector field over $M = \C^{\Z^2_{+,N}}$, and let $\mathcal{X}= \{\partial_{a_k},\partial_{b_k}\}_{k\in \mathcal{K}} \subseteq \mathfrak{X}(M)$ where $\mathcal{K}\subseteq \Z^{2}_{+,N}$ satisfies Assumption \ref{ass:hypo}.
	Then, $\forall w\in M$ (in a uniform way)
	\[
	\mathfrak{m}_w([\mathcal{X},X_0^\ep]) = \mathfrak{sl}(T_wM).
	\]
	and in particular, from Proposition \ref{prop:class-proj-span}, $\widetilde{X}_0^\ep,\widetilde{\mathcal{X}}$ satisfies the uniform-in-$\ep$ parabolic H\"ormander condition on $\S M$. 
\end{theorem} 
The majority of the paper is spent proving Theorem \ref{thm:Main-projective-spanning}; see Section \ref{sec:distinct} for a summary of how the pieces fit together in the proof.  
Next, we briefly summarize next why Theorem \ref{thm:Main-projective-spanning} implies Theorem \ref{thm:mainchaos} from the results of \cite{BBPS20}.  

The following lemma is a consequence of H\"ormander's theorem, Doob-Khasminskii's theorem, and geometric control theory; see \cite{BBPS20} for details. 
\begin{lemma}[Theorem B.1, \cite{BBPS20}]
Let $X_0^\eps(x) = B(x,x) - \eps Ax$ and consider the class of SDE \eqref{eq:xtclass} (with the corresponding conditions assumed on $B$).
Suppose that the lifts $\widetilde{X}_0^\ep,\widetilde{\mathcal{X}}$ satisfies the uniform parabolic H\"ormander condition on $\S M$. 
Then, $\forall \eps \in (0,\infty)$, there exists a unique stationary measure $\nu$ for the associated projective process with a smooth, strictly positive density $f^\eps$ with respect to Lebesgue measure such that $f^\eps \log f^\eps \in L^1$
and $\exists C,\gamma > 0$ such that $\forall \eps \in (0,1]$,
\begin{align*}
\int_{\mathbb S M } f^\eps e^{\gamma \abs{x}^2} dq < C, 
\end{align*}
and $\forall N > 0$, the following moment bound holds $\forall \eps \in (0,1]$ (not uniformly in $N$ or $\eps$),
\begin{align*}
\int_{\mathbb S M} \brak{x}^N f^\eps \log f^\eps \dee q < \infty. 
\end{align*}
\end{lemma}
In view of the above, Proposition \ref{prop:FI} gives the following for \eqref{eq:xtclass} (assuming projective hypoellipticity), 
\begin{align}
	FI(f^\ep ) = \frac{n \lambda_1}{\eps} - 2 \mathrm{tr} \, A. \label{eq:FInse}
\end{align}
Theorem \ref{thm:FIlow} then implies there exists an $s \in (0,1)$ such that for every bounded open set $U \subseteq SM$ we have
\begin{align*}
	\norm{f^\eps}_{W^{s_\star,1}(U)}^2 \leqc_U 1 + \frac{\lambda_1}{\eps}. 
\end{align*}
Therefore, if $\lambda_1/\eps$ were to remain bounded, one can show that $\set{f^\eps}_{\eps > 0}$ is precompact in $L^p$ for all $p \geq 1$ sufficiently small and so there is a strongly 
convergent subsequence $f^{\eps_n} \to f \in L^1$ which is an absolutely continuous stationary density for the $\eps = 0$ limiting deterministic projective process [Proposition 6.1, \cite{BBPS20}]. 

Projective hypoellipticity played the crucial role in reducing the estimate on the Lyapunov estimate to one of regularity of $f^\eps$, and for the estimate \eqref{eq:chaos} to whether or not there can exist an invariant measure with an $L^1$ density for the \emph{deterministic} $\eps = 0$ projective process.  
That no such invariant density can exist for any model of the form \eqref{eq:xtclass} was proved in [Proposition 6.2 \cite{BBPS20}], and therefore $\lambda_1/\eps \to \infty$. 
The major additional ingredient used in this step is that the $\eps = 0$ Jacobian $D_x \Phi^t$ grows unboundedly as $t \to \infty$ for a.e. initial condition $x$, necessitating concentrations in any invariant measures (see \cite{BBPS20} for details).  
This is deduced using the special structure of the nonlinearity and for other models may not be straightforward to verify.

To summarize, to prove Theorem \ref{thm:mainchaos}, it suffices only to verify that $\{\widetilde{X}^\eps_0, \widetilde{X}_1,...,\widetilde{X}_r\}$ satisfies the uniform-in-$\eps$ parabolic H\"ormander condition, i.e. Theorem \ref{thm:Main-projective-spanning}. 
The remainder of the paper is dedicated to the proving of this result, which as detailed in the next section, is a purely algebraic question. 

\begin{remark}
For the specific case of Navier-Stokes with additive stochastic forcing, the Fisher information becomes 
\begin{align*}
	FI(f^\ep ) = \sum_{k \in \mathcal{K}} \int \left(\abs{\partial_{a_k} \log f^\eps}^2 + \abs{\partial_{b_k} \log f^\eps}^2\right) f^\eps\, \dee q.
\end{align*}
\end{remark}

\section{Projective hypoellipticity on complex geometries} \label{sec:ProjComp}
%!TEX root = master.tex
\subsection{Real vs complex spanning}
Treating the phase space $\C^{\Z^2_{+,N}}$ as a real manifold using the real and imaginary parts can be awkward and lead to very cumbersome calculations. 
Due to the convenience of the Fourier description when dealing with Galerkin truncations of PDEs, it makes more sense to find a natural, complex way to view phase space. 
Specifically, if we have a complex phase space $\C^n$, we should treat it as a \emph{complex manifold} and complexify the tangent space.  
First, we review some of the basic concepts from complex geometry for the readers' convenience (see e.g. [Chapter 1, \cite{huybrechts2005complex}]) and explain how the ideas apply to hypoellipticity of stochastic PDEs, providing a cleaner proof of the spanning condition for Galerkin Navier-Stokes obtained in \cite{E2001-lg}.
Finally, we explain how the ideas extend to the question of projective hypoellipticity and formulate the sufficient condition which occupies the rest of the paper. 

\begin{definition}
Given a real vector space $V$, define its complexification by
\begin{align*}
V \otimes \C = \set{v_1 + i v_2 : v_1,v_2 \in V}. 
\end{align*}
\end{definition}
We begin by noting the following simple, but crucial equivalence between complex and real spanning of a collection of vectors in a real vector space.
\begin{lemma}\label{lem:complex-span}
Let $V$ be a real vector space and let $V\otimes\C$ be it's complexification. For a given collection of vectors $\{v_k\}\subset V$, we have
\begin{equation}\label{eq:complexspan}
	\Span\{v_k\} = V,\quad \text{if and only if}\quad \Span_{\C}\{v_k\} = V\otimes \C,
\end{equation}
where $\Span_{\C}$ denotes the span of a collection of vectors using complex coefficients.
\end{lemma}
\begin{proof}
Real spanning of $V$ implies complexified spanning since one can span the real and imaginary parts separately. For the converse, suppose that \eqref{eq:complexspan} holds. This means that for any $v\in V$ there exist $\{a_k\}\subset \C$ such that $\sum_{k} \alpha_k v_k = v$. Taking the real part of both sides gives $\sum_{k} \mathrm{Re}(\alpha_k) v_k = v$, implying that $\{v_k\}$ spans $V$.
\end{proof}

\subsection{H\"ormanders condition on $\C^n$}

Now we turn to the space $\C^n$, where complexification of the tangent space is natural and most useful. Let $X$ be a smooth vector field over $\C^n$, where $\C^n$ is viewed as a real manifold with real tangent space $T\C^n$ spanned by the coordinate vectors $\partial_{a_k}, \partial_{b_k}$, corresponding to the real and imaginary parts respectively. Clearly, $T\C^n$ is isomorphic as a vector space to $\C^n$ and therefore we may view each $X$ as a mapping $\C^n \to \C^n$ with $X^k:\C \to \C$ the $k$th component of the image of that map. In the $\partial_{a_k}, \partial_{b_k}$ basis we can write $X$ as
\begin{equation}\label{eq:real-form-vec}
	X = \sum_k\mathrm{Re}(X^k)\partial_{a_k} + \mathrm{Im}(X^k)\partial_{b_k}.
\end{equation}
In what follows we will complexify the tangent space $T\C^n\otimes \C$ and define complex basis vectors
\[
	\partial_{z_k} = \tfrac{1}{2}\left(\partial_{a_k} - i \partial_{b_k}\right),\quad \overline{\partial}_{z_k} = \tfrac{1}{2}\left(\partial_{a_k} + i \partial_{b_k}\right).
\]
This naturally induces the splitting $T\C^n\otimes \C = T^{1,0}\C^n \oplus T^{0,1}\C^n$, where $T^{1,0}\C^n = \Span_{\C}\{\partial_{z_k}\}$, $T^{0,1}\C^n = \Span_{\C}\{\overline{\partial}_{z_k}\}$, known as the holomorphic and anti-holomorphic bundles respectively (see \cite{huybrechts2005complex}). In this new basis, we see that \eqref{eq:real-form-vec} becomes
\begin{equation}\label{eq:complex-form-vec}
	X = \sum_kX^k\partial_{z_k} + \overline{X^k}\overline{\partial}_{z_k}.
\end{equation}

Recall that the Lie bracket $[\,\cdot\,,\,\cdot\,]$ is coordinate independent and does not depend on the choice of basis and so neither does $\mathrm{Lie}(\mathcal{F})$ for some collection $\mathcal{F}\subseteq \mathfrak{X}(\C^n)$. Given a collection $\mathcal{F}\subset \mathfrak{X}(\C^n)$, let $\mathrm{Lie}(\mathcal{F})_{\C}$ be the complexification of $\mathrm{Lie}(\mathcal{F})$ (obtained by replacing $\Span$ with $\Span_\C$ in the definition \eqref{eq:lie-def}). We now have the following simple corollary of Lemma \ref{eq:complexspan} regarding spanning for $\mathrm{Lie}_x(\mathcal{F})_\C:=\{ X(x)\,:\, X\in \mathrm{Lie}(\mathcal{F})_\C\}$.
\begin{lemma}\label{lem:complex-bracket-gen}
A collection $\mathcal{F}\subseteq \mathfrak{X}(\C^n)$ satisfies H\"ormander's condition on $\C^n$ (as a real manifold) if and only if for each $z\in \C^n$ 
\[
\mathrm{Lie}_z(\mathcal{F})_\C= T^{1,0}\C^n\oplus T^{0,1}\C^n.
\]
\end{lemma}

\begin{remark}The same proof also applies to any subalgebra of $\mathrm{Lie}(\mathcal{F})$, for instance the Lie algebra ideal $\mathrm{Lie}(X_0;\mathcal{F})$ with respect to a distinguished drift vector field $X_0$.
\end{remark}

\begin{remark}
Lemma \ref{lem:complex-bracket-gen} means from a practical perspective that in order to check H\"ormander's condition for a collection of vector fields on $\C^n$, it is sufficient to take complex linear combinations and attempt to isolate $\partial_{z_k}$ and $\overline{\partial}_{z_k}$ separately in order to span both $T^{1,0}\C^n$ and $T^{0,1}\C^n$.
\end{remark}

\subsection{Application: hypoellipticity for the stochastic Navier-Stokes equations}\label{subsec:SNS-app}

In this section we show how the complexification procedure above allows us to give a cleaner proof of H\"ormander's condition for the Navier-Stokes equations with additive stochastic forcing in 2d, first identified in \cite{E2001-lg} and expanded upon in \cite{HM06}.  
Recall from Section \ref{sec:GNSE}, we can formulate the 2d stochastic Galerkin-Navier-Stokes equations as an SDE on $M=\C^{\Z^2_{+,N}}$ by
\[
    \dot{w} = X^\ep_0(w) +  \sqrt{\eps} \sum_{k\in\mathcal{K}} \left(\alpha_k  \dot{W}_t^{(k;a)} \partial_{a_k} + \beta_k \dot{W}_t^{(\ell;b)} \partial_{b_k}\right),
\]
where $\alpha_k,\beta_k \in \R \neq 0$ for $k\in \mathcal{K}\subseteq \Z^2_{+,N}$ and $X_0^\ep(w) = B(w,w) + \ep A w$. In $\partial_{w_\ell}, \overline{\partial}_{w_\ell}$ coordinates it takes the form
\[
	X_0^\ep(w) = \sum_{\ell\in \Z^2_{+,N}}(B_{\ell}(w,w) - \ep |\ell|^2w_\ell)\partial_{w_\ell} + (\overline{B_\ell(w,w)} - \ep|\ell|^2 \overline{w_\ell})\overline{\partial}_{w_\ell},
\]
where
\begin{equation}\label{eq:complex-Euler}
	B_{\ell}(w,w) = \frac{1}{2}\sum_{j+k=\ell}c_{j,k}w_jw_k,
\end{equation}
with the sum over all $j,k\in \Z^2_{0,N}$ such that $j+k=\ell$. Due the reality constraint $w_{-\ell} = \overline{w}_{\ell}$ we find it convenient to index the basis vectors on the full lattice $\mathbb{Z}^2_0$ via
\[
	\partial_{w_\ell} := \begin{cases}
	\partial_{w_\ell} & \ell\in \Z^2_{+,N}\\
	\overline{\partial}_{w_{-\ell}} & \ell\in \Z^2_{-,N}
	\end{cases},
\]
where $\Z^2_- = -\Z^2_+$. Combining this with the reality constraint on $B_{-\ell}(w,w) = \overline{B_\ell(w,w)}$, we can write $X_0$ in a more succinct notation involving a sum over the full lattice
\[
	X_0^\ep(w) = \sum_{\ell\in \Z^2_0} (B_{\ell}(w,w) - \ep|\ell|^2w_\ell)\partial_{w_\ell}.
\]
We note that for any $w\in \C^{\Z^2_{0,N}}$, satisfying $w_{-\ell} = \overline{w}_{\ell}$, that $\partial_{w_\ell}$ as defined above has the property that for each $\ell,i\in \Z^2_{0,N}$, $\partial_{w_{\ell}}$ behave as Wirtinger derivatives, satisfying
\[
\partial_{w_\ell} w_{i} = \delta_{i=\ell}.
\]
From this, we can easily obtain simple expressions for the brackets
\[
	[\partial_{w_k},X_0^\ep(w)] = \sum_{j\in\Z^2_{0,N}}\1_{\Z^2_{0,N}}(j+k)c_{j,k}w_j\partial_{w_{j+k}} - \ep |k|^2\partial_{w_k},
\]
and
\[
	[\partial_{w_{k_1}}, [ \partial_{w_{k_2}}, X_0^\ep(w)]] = \1_{\Z^2_{0,N}}(k_1+k_2)c_{k_1,k_2}\partial_{w_{k_1+k_2}}.
\]

Our goal is to prove the following:
\begin{proposition}\label{prop:SNS-bracket-theorem}
Let $\mathcal{X}=\{\partial_{a_k},\partial_{b_k}\,:\,k\in\mathcal{K}\}$, where $\mathcal{K}\subseteq \Z^{2}_{+,N}$ satisfies Assumption \ref{ass:hypo}. 
Then $\mathrm{Lie}(X_0^\ep;\mathcal{X})_\C$ contains the constant vector fields $\{\partial_{w_k}\,:\, k\in\Z^2_{0,N}\}$ and moreover, it follows from Lemma \ref{lem:complex-bracket-gen} that $X_0^\ep,\mathcal{X}$ satisfies the uniform parabolic H\"ormander condition on $\C^{\Z^2_{+,N}}$ viewed as a real manifold. 
\end{proposition}
\begin{proof}
Since for a given $\ell \in \Z^2_{+,N}$, $\partial_{w_\ell}$ and $\partial_{w_{-\ell}}= \overline{\partial}_{w_\ell}$ are complex linear combinations of $\partial_{a_\ell}$ and $\partial_{b_\ell}$ for $\ell \in\Z^2_{+,N}$ it suffices to take brackets with respect to $\partial_{w_\ell}$ for all $\ell\in \mathcal{K}\cup\{-\mathcal{K}\}\subseteq \Z^2_{0,N}$. 
Therefore for $k_1,k_2\in \mathcal{K}\cup\{-\mathcal{K}\}$ we have
\[
	[\partial_{w_{k_1}},[\partial_{w_{k_2}},X_0]] = 1_{\Z^2_{0,N}}(k_1+k_2)c_{k_1,k_2}\partial_{w_{k_1+k_2}};
\]
as this is independent of $\eps$, it is clear that spanning will imply uniform spanning. 
If $c_{k_1,k_2}\neq 0$ we conclude that $\partial_{w_{k_1+k_2}} \in \mathrm{Lie}(X_0;\mathcal{F})_\C$.

It becomes clear we need the following iteration, defining $\mathcal{Z}_0 = \mathcal{K} \cup\{-\mathcal{K}\}$
\begin{align*}
\mathcal{Z}_n = \set{\ell + j : j \in \mathcal{Z}_0, \, \ell \in \mathcal{Z}_{n-1} \textup{ such that } c_{\ell,j} \neq 0 }
\end{align*}
By Assumption \ref{ass:hypo}, this iteration continues to generate all of $\mathbb Z_{0,N}^2$ which implies that $\{\partial_{w_k}\}_{k\in\Z^2_{0,N}} \subseteq \mathrm{Lie}(X_0;\mathcal{F})_\C$ and therefore, since $T^{0,1}\C^{\Z^2_{0,N}} \simeq T^{0,1}\C^{\Z^2_{+,N}}\oplus T^{1,0}\C^{\Z^2_{+,N}} \subseteq \mathrm{Lie}_z(X_0;\mathcal{F})_\C$, the theorem is proved. 
\end{proof}

\subsection{Projective spanning on $\C^n$}

When the manifold is $\C^n$ we will also find it useful to complexify the tangent space to show projective hypoellipticity. 
Let $V$ be a real vector space and recall that for a given vector space $W$ (real or complex) the space $\mathfrak{sl}(W)$ is the Lie algebra of linear endomorphisms $H$ of $W$ with $\tr H = 0$ (note this is independent of basis) and Lie bracket given by the commutator
\[
	[A,B] = AB - BA.
\]
Note that any endomorphism $H$ of $V$ can be trivially extended to an endomorphism of the complexification $V\tensor \C$ via $H(v_1+iv_2) = Hv_1 + i Hv_2$, moreover any $G\in \mathfrak{sl}(V\tensor \C)$ can be written as $G = G_1 + i G_2$, where $G_1,G_2\in\mathfrak{sl}(V)$, so that we have $\mathfrak{sl}(V\tensor \C) = \mathfrak{sl}(V)\tensor \C$, i.e. $\mathfrak{sl}(V)$ is a {\em real form} for $\mathfrak{sl}(V\tensor \C)$. We denote the Lie algebra of endomorphisms generated by any collection $\mathcal{H}\subseteq \mathfrak{sl}(V)$  by
\[
	\mathrm{Lie}(\mathcal{H}) = \Span\{\mathrm{ad}(H^{r})\ldots\mathrm{ad}(H^{2})H^{1}
	\,:\, H_i\in \mathcal{H},\,r\in \N\}.
\]
Likewise, define $\mathrm{Lie}(\mathcal{H})_\C$ as above with $\mathrm{span}_\C$ and $\mathcal{H}$ extended to $\mathfrak{sl}(V\tensor\C)$. The next result follows easily from Lemma \ref{lem:complex-span} and the bilinearity of $X,Y\mapsto \mathrm{ad}(X)Y$.
\begin{proposition}\label{prop:sln-span}
Let $V$ be a real vector space and $\mathcal{H}\subseteq \mathfrak{sl}(V)$, then $\mathrm{Lie}(\mathcal{H}) = \mathfrak{sl}(V)$ if and only if $\mathrm{Lie}(\mathcal{H})_{\C} = \mathfrak{sl}(V)\tensor \C$.

\end{proposition}

In light of the linearity of the mapping $X\mapsto M_X(z)$ we have the following property of the Lie algebra of endomorphisms induced by $\mathrm{Lie}(\mathcal{F})_\C$ for some collection $\mathcal{F}\in \mathfrak{X}(\C^n)$.
\[
	\mathfrak{m}_z(\mathcal{F})_\C := \{M_X(z) \,:\, X\in \mathrm{Lie}(\mathcal{F})_\C\,,\, X(z) =0\}.
\]
\begin{corollary}\label{cor:mx-span}
Let $\mathcal{F}\subseteq \mathfrak{X}(\C^n)$, then for each $z\in \C^n$ we have $\mathfrak{m}_z(\mathcal{F}) = \mathfrak{sl}(T_z\C^n)$ if and only if $\mathfrak{m}_z(\mathcal{F})_\C = \mathfrak{sl}(T_z\C^n)\tensor\C$.
In particular, the lifts $\widetilde{\mathcal{F}}$ satisfies H\"ormander's condition on $\S \C^n$ if $\mathfrak{m}_z(\mathcal{F})_\C = \mathfrak{sl}(T_z\C^n)\tensor \C$ and $\mathcal{F}$ satisfies 
H\"ormander's condition on $\C^n$.  
\end{corollary}

\begin{remark}
Corollary \ref{cor:mx-span} is useful in the sense that it allows one to work directly with $\mathfrak{m}_x(X_0;\mathcal{F})_\C$ therefore consider matrices $\nabla X(x)$ in $\partial_{z_k}$, $\overline{\partial}_{z_k}$ coordinates, which often take a much simpler form than their counterparts in $\partial_{a_k},\partial_{b_k}$ coordinates.
\end{remark}

\subsection{A sufficient condition for projective hypoellipticity for Navier-Stokes}

In this section, we consider a sufficient condition for projective hypoellipticity for the Navier-Stokes equation in terms of a real matrix Lie algebra obtained by working in complex coordinates. These matrices take on a particularly simple form that allow the problem to be made much more tractable which is crucial for the arguments that follow.

Following the set-up of section \ref{subsec:SNS-app}, we define the Navier-Stokes vector field on the complexified tangent space $T_w\C^{\Z^2_{+,N}} \otimes \C$
\[
	X_0^\ep(w) := \sum_{\ell\in\Z^2_0} \left(B_\ell(w,w) - \ep|\ell|^2w_\ell\right) \partial_{w_\ell},
\]
where we recall that $w_{-\ell} = \overline{w}_\ell$ and that we have defined for $\ell\in\Z^2_0$
\[
	\partial_{w_\ell} := \begin{cases}
	\partial_{w_\ell} & \ell\in \Z^2_+\\
	\overline{\partial}_{w_{-\ell}} & \ell\in \Z^2_-
	\end{cases}.
\]
As in the set up of Proposition \ref{prop:SNS-bracket-theorem}, we assume that we have vector fields $\{\partial_{w_k}\}_{k\in \mathcal{K} \cup -\mathcal{K}}$, where $\mathcal{K}\subset \Z^2_{0,N}$ generates $\Z^2_{0,N}$ in the sense of Assumption \ref{ass:hypo}. 
By Proposition \ref{prop:SNS-bracket-theorem}, we have that $\mathrm{Lie}(X_0;\{\partial_{w_k}\}_{k\in\mathcal{K}})_\C$, contains the constant vector fields $\{\partial_{w_k}\}_{k\in \Z^2_{0,N}}$, and therefore any vector field $X\in \mathrm{Lie}(X_0;\{\partial_{w_k}\}_{k\in\mathcal{K}})_\C$ can always be shifted by a constant vector field
\[
\hat{X} = X- X(z)
\]
so that $\hat{X}(z) = 0$ and $\nabla \hat{X} = \nabla X$. Additionally, by Corollary \ref{cor:mx-span}, and the fact that $B(w,w)$ is bilinear and $\Delta w$ is linear, this implies that for each $k\in\Z^2_{0,N}$, the endomorphism
\[
	H^k := \nabla [\partial_{w_k},X_0^\ep] = \partial_{w_k}\nabla B,
\]
belongs to $\mathfrak{m}_w(X_0;\{\partial_{w_k}\}_{k\in\mathcal{K}})_\C$. 
Moreover due to the bilinear nature of $B(w,w)$, each $H^k$ is constant and independent of $\ep$.

The following Lemma gives an explicit matrix representation of $H^k$ in $\partial_{w_k}$ coordinates as a $|\Z^2_{0,N}| = (2N+1)^2-1$ dimensional square matrix indexed over $\Z^2_{0,N}$. This simple form comes from the convenient form of the nonlinearity in complex variables (see Section \ref{sec:GNSE}). 
\begin{lemma}\label{lem:Hk-Gk-form}
For each $k\in \Z^2_{0,N}$ we have the following formula for $H^k$ in $\partial_{w_\ell}$ coordinates by
\begin{equation}\label{eq:Hk-Def}
(H^k)_{\ell,j} = c_{j,k}\delta_{k+j=\ell}, \quad \ell,j\in \Z^2_{0,N}.
\end{equation}
\end{lemma}

Note that in $\{\partial_{w_k}\}$ coordinates, the matrices $H^k$ are {\em real} matrices, and therefore we only need them to generate an appropriate Lie algebra of real matrices in order for them to span the complexified space by Corollary \ref{cor:mx-span}. 

Below we record a sufficient condition for projective spanning in the Galerkin-Navier-Stokes system in terms of the Lie algebra generated by the $H^k$ matrices.
\begin{proposition}\label{prop:sufficient-projectspan-NS}
Let $\{H^k\}:= \{H^k\,:\, k\in Z^2_{0,N}\}$ 
be the matrices defined by \eqref{eq:Hk-Def} in $\partial_{w_k}$ coordinates.
Then the lifts $\widetilde{X}_0^\ep$, $\{\widetilde{\partial}_{a_k},\widetilde{\partial}_{b_k}\,:\,k\in\mathcal{K}\}$ satisfy the uniform parabolic H\"ormander condition on $\S \C^{\Z^2_{+,N}}$ if
\[
	\mathrm{Lie}(\{H^k\}) = \mathfrak{sl}_{\Z^2_{0,N}}(\R), 
\]
where we use the notation $\mathfrak{sl}_{\Z^2_{0,N}}(\R)$ to denote the Lie algebra of real, trace-free matrices indexed by $\Z^2_{0,N}$. 
\end{proposition}
\begin{proof}
By the above discussion regarding Proposition \ref{prop:SNS-bracket-theorem}, we have that $\{H^k\}$ viewed as linear endomorphisms satisfy 
\[
\{H^k\} \subset \mathfrak{m}_w(X_0;\{\partial_{w_k}\}_{k\in\mathcal{K}\cup(-\mathcal{K})})_\C \subseteq \mathfrak{sl}(T_{w}\C^{\Z^2_{+,N}})\tensor \C.
\]
If the corresponding real matrix Lie algebra $\mathrm{Lie}(\{H^k\})$ represented in $\partial_{w_k}$ coordinates is equal to $\mathfrak{sl}_{\Z^2_{0,N}}(\R)$, then by Proposition \ref{prop:sln-span} it is clear that the complexified algebra of endomorphisms satisfies $\mathrm{Lie}(\{H^k\})\tensor \C = \mathfrak{sl}(T_{w}\C^{\Z^2_{+,N}})\tensor \C$ and therefore 
\[
	\mathfrak{m}_w(X_0;\{\partial_{w_k}\}_{k\in\mathcal{K}\cup(-\mathcal{K})})_\C = \mathfrak{sl}(T_{w}\C^{\Z^2_{+,N}})\tensor \C.
\] 
Moreover, since $\{H^k\}$ are constant matrices, this equality is {\em uniform} in $w$.
It follows by Corollary \ref{cor:mx-span} that $\widetilde{X}_0^\ep$, $\{\widetilde{\partial}_{a_k},\widetilde{\partial}_{b_k}\,:\,k\in\mathcal{K}\}$ satisfy the uniform parabolic H\"ormander condition on $\S \C^{\Z^2_{+,N}}$.
\end{proof}

\begin{remark}
It is important to note that each matrix $H^k$ has a banded structure, with non-zero entries occurring on the band $\ell -j = k$ (except when $c_{j,k} = 0$). This banded structure is a consequence of the non-local frequency coupling present in non-linearity. Such non-local interactions provide a significant challenge when studying the $\mathrm{Lie}(\{H^k\})$ and are what make projective spanning for Navier-Stokes and other PDEs so challenging compared to locally coupled models like Lorenz 96 or shell models like GOY or SABRA.
\end{remark}

\section{A dinstinctness condition in the diagonal algebra} \label{sec:diagonal}
%!TEX root = master.tex

In order to show that 
\[
\mathrm{Lie}(\{H^k\}) = \mathfrak{sl}_{\Z^2_{0,N}}(\R),
\]
for Proposition \ref{prop:sufficient-projectspan-NS}, a special role will be played by a certain diagonal subalgebra $\mathfrak{h}$. Genericity properties of elements of this algebra, specifically related to distinctness of certain differences of diagonal elements, will play a crucial role in our ability to isolate elementary matrices, which is particularly challenging for the Navier-Stokes equations due to the non-local frequency interactions made explicit by the banded structure of the $H^k$ matrices.

\subsection{An illustrative example} \label{sec:ill}

Taking a page from the classical root space decomposition of semi-simple Lie algebras, we will make use of a strategy that utilizes the fact that elementary matrices are left invariant by adjoint action with a diagonal matrix. To fix ideas, we will first consider an idealized situation. Let $\mathbb{D}$ be any diagonal matrix in $\mathfrak{sl}_n(\R)$ with diagonal entries $\mathbb{D}_{ii} $ denoted by $\mathbb{D}_i$. It is well known and easily verifiable that for any elementary matrix $E^{i,j} = \delta_{ij}$ for the Kronecker delta (i.e. a matrix with a one in the $i$th row and $j$th column and zero elsewhere) one has
\[
	\mathrm{ad}(\mathbb{D})E^{i,j} = [\mathbb{D},E^{i,j}] = (\mathbb{D}_i - \mathbb{D}_j)E^{i,j},
\]
and therefore $E^{i,j}$ is an {\em eigenvector} of the operator $\mathrm{ad}(\mathbb{D})$ with eigenvalue $\mathbb{D}_i - \mathbb{D}_j$. This means that if $\mathbb{D}$ is suitably generic in the sense that it's diagonal entries have {\em distinct differences}
\[
	\mathbb{D}_i - \mathbb{D}_j \neq \mathbb{D}_{i\prime} - \mathbb{D}_{j^\prime}\quad \text{when } (i,j) \neq (i^\prime,j^\prime),
\]
then the operator $\mathrm{ad}(\mathbb{D})$ has simple eigenvalues. Such a distinctness property and associated simplicity of the spectrum gives a clear strategy for spanning sets of elementary matrices that generate $\mathrm{sl}_n(\R)$ using an approach similar to Krylov subspace methods for generating sets of linearly independent eigenvectors \cite{Arnoldi1951,Trefethen1997}. Specifically we have the following. 

\begin{proposition}\label{prop:simple-power-method}
Let $H$ be a matrix in $\mathfrak{sl}_{n}(\R)$ whose diagonal entries are zero $H_{ii} = 0$, and with at least one non-zero element away from the diagonal, 
\[
	\mathrm{supp}(H) := \{(i,j)\,:\, i\neq j\,, H_{ij}\neq 0\} \neq \emptyset.
\]
Suppose, in addition, that there is a diagonal matrix $\mathbb{D}$ whose diagonal entries $\mathbb{D}_i = \mathbb{D}_{ii}$ satisfy
\[
	\mathbb{D}_i - \mathbb{D}_j \neq \mathbb{D}_{i^\prime} - \mathbb{D}_{j^\prime},\quad  \text{ for each}\quad (i,j),(i^\prime,j^\prime) \in \mathrm{supp}(H), \quad (i,j)\neq (i^\prime,j^\prime).
\]
Then for $N = |\mathrm{supp}(H)|$, 
\[
	\mathrm{span}\{H,\mathrm{ad}(\mathbb{D})H,\mathrm{ad}(\mathbb{D})^2H\ldots,\mathrm{ad}(\mathbb{D})^{N-1}H\}, \quad 
\]
contains the elementary matrices $E^{i,j}$ for each $i,j\in \mathrm{supp}(H)$.
\end{proposition}
\begin{proof}
Write 
\begin{align*}
H = \sum_{H_{ij} \neq 0} H_{ij}E^{ij}
\end{align*}
and $\lambda_{ij} = \D_{i} - \D_j$ be the corresponding eigenvalue of $\mathrm{ad}(\D)$. 
The Krylov subspace in question becomes 
\begin{align*}
\mathrm{span}\set{ \sum H_{ij}E^{ij}, \sum H_{ij} \lambda_{ij} E^{ij}, \ldots  \sum H_{ij} \lambda_{ij}^{N-1} E^{ij} }.  
\end{align*}
The linear independence of these vectors reduces to the invertibility of the Vandermonde matrix 
\begin{align*}
\begin{pmatrix}
1 & \lambda_{i_1 j_1} & \ldots & \lambda_{i_1 j_1}^{N-1} \\ 
\vdots  & \vdots & \ddots & \vdots \\ 
 1 & \lambda_{i_N j_N} & \ldots & \lambda_{i_N j_N}^{N-1}
\end{pmatrix}, 
\end{align*}
which follows by the assumption that $\lambda_{ij} \neq \lambda_{i'j'}$ for $(i,j) \neq (i',j')$. 
Hence, due to the full rank, the Krylov subspace coincides with $\mathrm{span}\set{E^{ij} : H_{ij} \neq 0}$. 
\end{proof}

\begin{remark}\label{rem:generating-set}
In the context of Lemma \ref{prop:simple-power-method}, it is important to note that one does not necessarily need $H$ to have all it's entries non-zero in order to show that $\mathrm{Lie}(\mathbb{D},H) = \mathfrak{sl}_n(\R)$. Indeed, a relatively small number of elementary matrices can easily generate $\mathfrak{sl}_n(\R)$. For instance it is readily seen that the elementary matrices
\[
	E^{1,2},E^{2,3},\ldots E^{n-1,n},E^{n,1}
\]
are sufficient to generate $\mathfrak{sl}_n(\R)$.
\end{remark}

\subsection{The diagonal subalgebra $\mathfrak{h}$}

The set of matrices $\{H^k\}$ defined in \eqref{eq:Hk-Def} does not contain any diagonal matrices, however, by commuting $H^k$ and $H^{-k}$, we obtain a diagonal algebra which we denote
\[
\mathfrak{h} := \Span\{[H^{k},H^{-k}]\,:\, k\in\Z^2_{0,N}\}.
\]
\begin{lemma}
For each $k\in \Z^2_{0,N}$, we have
\begin{equation}\label{eq:block-diag-matrix}
\mathbb{D}^k := [H^k,H^{-k}]
\end{equation}
is a diagonal matrix with diagonal entries for each $i \in \Z^2_{0,N}$ given by
\begin{equation}\label{eq:Diag-def}
	\mathbb{D}^k_i := c_{i,k}c_{i+k,k}\1_{\Z^2_{0,N}}(i+k) - c_{i,k}c_{i-k,k} \1_{\Z^2_{0,N}}(i-k),
\end{equation}
and therefore $\mathfrak{h} = \Span\{\mathbb{D}^k\}$ is a commutative Lie sub-algebra of $\mathfrak{sl}_{\Z^2_{0,N}}(\R)$.
\end{lemma}

\begin{remark}\label{rem:infinite-algebra-diag}
It is important to note that the truncated lattice $\Z^2_{0,N}$ actually makes the form of $\mathbb{D}^k_i$ more complicated. Depending on the choice of $k$, the indicator functions $\1_{\Z^2_{0,N}}(i+k)$ and $\1_{\Z^2_{0,N}}(i-k)$ have non-trivial regions where they overlap and don't overlap, leading to significant complications in proofs that utilize computational algebra. A remarkable fact is that the ``infinite dimensional'' case obtained by replacing $\Z^2_{0,N}$ with the full lattice $\Z^2_{0}$ actually 
gives the much cleaner form
\[
	\mathbb{D}^{k}_i = c_{i,k}c_{i+k,k} - c_{i,k}c_{i-k,k} = \langle i^\perp, k\rangle^2\left(\frac{1}{|k|^2} - \frac{1}{|i|^2}\right)\left(\frac{1}{|i-k|^2} - \frac{1}{|i+k|^2}\right)
\]
making it much more amenable to algebraic methods.
\end{remark}

We would like to use the diagonal matrices in $\mathfrak{h}$ 
to proceed as Section \ref{sec:ill}, however, the situation here is far more delicate than that presented in Proposition \ref{prop:simple-power-method} due to the fact that $\mathbb{D}^k_i$ has an inversion symmetry $\mathbb{D}^k_{-i} = -\mathbb{D}^k_{i}$, which fundamentally restricts the possibility of having distinct differences. In particular, for any given diagonal matrix $\mathbb{D}$ satisfying $\mathbb{D}_{-i}=-\mathbb{D}_i$, the adjoint operator
\begin{equation}\label{eq:adjoint-block}
\mathrm{ad}(\mathbb{D}): \mathfrak{sl}_{\Z^2_{0,N}}(\R) \to \mathfrak{sl}_{\Z^2_{0,N}}(\R)
\end{equation}
is incapable of having simple spectrum since the odd symmetry of $\mathbb{D}_i$ implies that there are always two dimensional invariant spaces associated to the adjoint operator. Specifically we see that for each $i,j\in\Z^2_{0,N}$, $i\neq j$
\[
	\mathrm{ad}(\mathbb{D})E^{i,j} = (\mathbb{D}_i-\mathbb{D}_j)E^{i,j} \quad \text{and}\quad \mathrm{ad}(\mathbb{D})E^{-j,-i} = (\mathbb{D}_i-\mathbb{D}_j)E^{-j,-i}
\]
and therefore the eigenvalue $\mathbb{D}_i - \mathbb{D}_j$ for $\mathrm{ad}(\mathbb{D})$always has multiplicity at least $2$ with the invariant space
\[
	\mathrm{span}\{E^{i,j},E^{-j,i}\}.
\]

With this in mind, it is convenient to write $H^k$ as a linear combination of such matrices. In particular we can write for each $k\in\Z^2_{0,N}$
\[
\begin{aligned}
	H^k
	 = \frac{1}{2}\sum_{i,i-k\in \Z^{2}_{0,N}}(c_{i-k,k}E^{i,i-k} - c_{i,k}E^{k-i,-i}).
	 \end{aligned}
\]
Taking into account the sparsity of $H^k$ and the fact that for any diagonal matrix $\mathbb{D}$ satisfying $\mathbb{D}_{-i} = -\mathbb{D}_i$, $\mathrm{ad}(\mathbb{D})$ leaves $c_{i-k,k}E^{i,i-k} - c_{i,k}E^{k-i,i}$ invariant, suggests that if $\mathbb{D}$ satisfies the following distinctness property
\begin{equation}\label{eq:Distinctness}
\mathbb{D}_{i} - \mathbb{D}_{i-k} \neq \mathbb{D}_{i^\prime} - \mathbb{D}_{i^\prime - k},
\end{equation}
for each $k,i,i^\prime,i-k, i^\prime - k \in \Z^2_{0,N}$ with $i\neq i^\prime$ and $i \neq k - i^\prime$, then a similar procedure to the one carried out in Proposition \ref{prop:simple-power-method} implies that under the distinctness condition \eqref{eq:Distinctness}, if $\mathbb{D}$ belongs to $\mathrm{Lie}(\{H^k\})$, then $\mathrm{Lie}(\{H^k\})$ also contains the following sets of matrices for each $k \in \Z^2_{0,N}$ and $i,i-k \in \Z^2_{0,N}$
\[
	c_{i-k,k}E^{i,i-k} - c_{i,k}E^{k-i,-i}.
\]
By relabeling indices and eliminating $k$, this means that we can obtain matrices of the form
\begin{equation}\label{eq:M-matrix}
	M^{i,j} := c_{j,i-j}E^{i,j} - c_{i,i-j}E^{-j,-i}
\end{equation}
for each $i,j \in \Z^2_{0,N}$, $i-j\in\Z^2_{0,N}$. Similarly, in the distinctness condition \eqref{eq:Distinctness} we can eliminate $k$, and reduce this to a more symmetric constraint of the form
\[
	\mathbb{D}_i^k + \mathbb{D}_j^k + \mathbb{D}_{\ell}^k + \mathbb{D}_{m}^k \neq 0
\]
for $i,j,\ell,m\in \Z^2_{0,N}$ satisfying
\begin{equation}\label{eq:sum-constraint}
	i+ j + \ell+ m = 0,
\end{equation}
with the constraints
\begin{equation}\label{eq:constraint-set}
	\mathcal{C}_N: = \{ (i,j,\ell,m) \in (\Z^2_{0,N})^4\,:\,(i+j,\ell +m) \neq 0,\, (i+ \ell,j+m)\neq 0,\, (i+m,j + \ell) \neq 0\}.
\end{equation}
In general we have the following convenient reformulation of the distinctness condition \eqref{eq:Distinctness}.
\begin{definition}[Distinct]\label{def:distinctless-reform} We say a diagonal matrix $\mathbb{D} \in \mathfrak{h}$ is {\em distinct} if for every $(i,j,\ell,m)\in \mathcal{C}_N$ with $i+j+\ell +m =0$ we have
\begin{equation}\label{eq:symm-distinct}
	\mathbb{D}_i + \mathbb{D}_j + \mathbb{D}_{\ell} + \mathbb{D}_m \neq 0.
\end{equation}
\end{definition}

\begin{remark}
Note that the constraint set $\mathcal{C}_N$ defined in \eqref{eq:constraint-set} is fundamental to the symmetry of the sum $\mathbb{D}_i + \mathbb{D}_j + \mathbb{D}_{\ell} + \mathbb{D}_m$. Each constraint is necessary in the sense that if any one of them fails then we automatically have
\[
	\mathbb{D}_i + \mathbb{D}_j + \mathbb{D}_{\ell} + \mathbb{D}_{m} = 0
\]
due to the inversion symmetry $\mathbb{D}_{-i}= -\mathbb{D}_i$. 
\end{remark}

Under this new definition, we can summarize the above discussion as follows.

\begin{lemma}\label{lem:Mij-gen}
Suppose that $\mathfrak{h}$ contains a distinct diagonal matrix in the sense of Definition \ref{def:distinctless-reform}. Then $\mathrm{Lie}(\{H^k\})$ contains the matrices $\{M^{i,j}\,:\,i,j\in\Z^2_{0,N},\, i- j\in \Z^2_{0,N}\}$.
\end{lemma}

It turns out that this set of matrices $M^{i,j}$, each one being comprised of linear combination of pairs of elementary matrices, is sufficient to generate all of $\mathfrak{sl}_{\Z^{2}_{0,N}}(\R)$. The proof of this fact is the content of the following subsection.
\begin{proposition}\label{prop-Mij-generate}
The matrices $\{M^{i,j}\,:\,i,j\in\Z^2_{0,N},\, i- j\in \Z^2_{0,N}\}$ generate $\mathfrak{sl}_{\Z^{2}_{0,N}}(\R)$.
\end{proposition}

As a simple corollary of this and Lemma \ref{lem:Mij-gen} this reduces Proposition \ref{prop:sufficient-projectspan-NS} to a condition on the existence of a distinct matrix inside $\mathfrak{h}$.

\begin{corollary} \label{cor:disDon}
If $\mathfrak{h}$ contains a distinct matrix, then $\mathrm{Lie}(\{H_k\}) = \mathfrak{sl}_{\Z^{2}_{0,N}}(\R)$.
\end{corollary}

\subsubsection{Proof of Proposition \ref{prop-Mij-generate}}

To show Proposition \ref{prop-Mij-generate} we first assume an algebraic property of the coefficients $c_{j,k}$, which we will prove in Proposition \ref{prop:algebraic-geo-proof-1} using techniques from computational algebraic geometry. To simplify notation in what follows, denote
\[
	\mathcal{S}^i := \Z^2_{0,N}\cap\{\Z^2_{0,N}+i\}.
\] 

\begin{lemma}
Suppose that for each $i,j\in \Z^2_{0,N}$, with $i-j\in \Z^2_{0,1}$, there exists a $k,k^\prime\in \mathcal{S}^i\cap\mathcal{S}^j$ such that
\[
	d_{i,j}^{k,k^\prime} :=  c_{i,i-k}c_{k,j-k}c_{k^\prime,i-k^\prime}c_{j,j-k^\prime}-c_{k,i-k}c_{j,j-k}c_{i,i-k^\prime}c_{k^\prime,j-k^\prime} \neq 0.
\]
Then $\{M^{i,j}\,:\, i,j\in \Z^2_{0,N}, i- j\in \Z^2_{0,N}\}$ generates $\mathfrak{sl}_{\Z^2_{0,N}}(\R)$.
\end{lemma}
\begin{proof}
If we take commutators of matrices of the form $[M^{i,k},M^{k,j}]$, where $i-j\in\Z^2_{0,1}$ and $k\in \mathcal{S}^{i}\cap\mathcal{S}^j$, we have 
\[
	[M^{i,k},M^{k,j}] = c_{k,i-k}c_{j,k-j}E^{i,j} - c_{i,k-i}c_{k,j-k}E^{-j,-i}.
\]
Therefore if we pick any two $k,k^\prime \in \mathcal{S}^{i}\cap\mathcal{S}^j$, we obtain a $2\times 2$ linear system for $E^{i,j}$ and $E^{-j,-i}$
\begin{equation}\label{eq:Linear-sys-Eij}
\begin{aligned}
[M^{i,k},M^{k,j}] &= c_{k,i-k}c_{j,k-j}E^{i,j} - c_{i,k-i}c_{k,j-k}E^{-j,-i}\\
[M^{i,k^\prime},M^{k^\prime,j}] &= c_{k^\prime,i-k^\prime}c_{j,k^\prime-j}E^{i,j} - c_{i,k^\prime-i}c_{k^\prime,j-k^\prime}E^{-j,-i}.
\end{aligned}
\end{equation}
We can write $E^{i,j}$ and $E^{-j,-i}$ as a linear combination of $[M^{i,k},M^{k,j}]$ and $[M^{i,k^\prime},M^{k^\prime,j}]$
provided that for each $i,j\in \Z^2_0$, with $i-j\in\Z^2_{0,1}$ we can find a $k,k^\prime\in \mathcal{S}^i\cap\mathcal{S}^j$ such that
\[
	d^{k,k^\prime}_{i,j} = \mathrm{det} \begin{pmatrix}-c_{k,i-k}c_{j,j-k} & c_{i,i-k}c_{k,j-k}\\ -c_{k^\prime,i-k^\prime}c_{j,j-k^\prime} & c_{i,i-k^\prime}c_{k^\prime,j-k^\prime}
	\end{pmatrix} \neq 0.
\]
This is true by assumption and therefore we can solve the linear system \eqref{eq:Linear-sys-Eij} and obtain all elementary matrices $E^{i,j}$ for $i,j\in \Z^2_{0,N}$, with $i-j\in\Z^2_{0,1}$.  
One can easily check that all such elementary matrices generate $\mathfrak{sl}_{\Z^2_{0,N}}(\R)$ (see Remark \ref{rem:generating-set}).
\end{proof}

Note that the property that $d_{i,j}^{k,k^\prime}\neq 0$ is a purely algebraic one. In particular, suppose by contradiction, that there exists an $i,j\in \Z^2_{0,N}$ with $i-j\in \Z^{2}_{0,1}$ such that
\[
	d_{i,j}^{k,k^\prime} = 0,  \quad \text{for all}\quad k,k^\prime \in \mathcal{S}^i\cap\mathcal{S}^j.  
\]
That is then $i,j$ must solve a set of rational equations (with integer coefficients), one for each pair $(k,k') \in (\mathcal{S}^i\cap\mathcal{S}^j)^2$. 
Next, we show that this system of rational equations is algebraically inconsistent. We will prove the following Proposition using machinery from computational algebraic geometry, which we review in Appendix \ref{sec:Groebner}.
The computations are done using Maple; see Appendix \ref{subsec:alg-geo-proof-1} for the computer code. 
\begin{proposition}\label{prop:algebraic-geo-proof-1}
For each $i,j\in \Z^2_{0,N}$ with $i-j\in \Z^{2}_{0,1}$, there exists $k,k^\prime \in \mathcal{S}^i\cap\mathcal{S}^j$ such that
$d^{k,k^\prime}_{i,j} \neq 0$.
\end{proposition}
\begin{proof}
To prove this, we first note that $d(i,j,k,k^\prime,r) = d^{k,k^\prime}_{i,j}$ is a purely rational algebraic function of the variables $i = (i_1,i_2),j=(j_1,j_2),k=(k_1,k_2), k^\prime = (k^\prime_1,k^\prime_2)$ and $r$. We denote the numerator by
\[
	P(i,j,k,k^\prime,r) = \text{numer}(d(i,j,k,k^\prime,r)).
\]
Suppose by contradiction that there exists an $i,j \in \Z^2_{0,N}$ with $i-j\in \Z^2_{0,1}$ such that $d(i,j,k,k^\prime,r) = 0$ for all $k,k^\prime\in\mathcal{S}^i\cap\mathcal{S}^j$. Then we have
\[
	P(i,j,k,k^\prime,r) = 0 , \quad \text{for all}\quad k,k^\prime \in \mathcal{S}^i\cap\mathcal{S}^j.
\]
Note that this polynomial is degree $10$ in $k,k^\prime$. 
If $N \geq 8$ then since $i-j \in \Z^2_{0,1}$, $\mathcal{S}^i\cap\mathcal{S}^j$ always contains $\Z^2_{0,6}$ and Lemma \ref{lem:Reducek} implies that the collection of polynomials in $i,j,r$ defined by the coefficients of the polynomial in $k_1,k_2,k^\prime_1,k^\prime_2$
\[
	\{f_1,\ldots,f_s\} = \mathrm{coeffs}(P,\{k_1,k_2,k^\prime_1,k^\prime_2\})
\]
must also vanish (note that the coefficients $f_1,...,f_s$ are polynomials in $(i_1,i_2,j_1,j_2)$ and $r$ with integer coefficients). 
By extending the variables $i_1,i_2,j_1,j_2$ and $r$, to the algebraically closed field $\C$, we define the polynomial ideal generated by $\{f_1,\ldots,f_s\}$ 
\[
	I = \langle f_1,\ldots,f_s\rangle\subseteq \C[i_1,i_2,j_1,j_2,r]
\] 
(see Appendix \ref{sec:Groebner} for a review of the relevant algebraic geometry). 
Next we define the constraint polynomial
\[
	g(i,j,r) = r^2|i|_r^2|j|_r^2|i-j|_r^2.
\]
Note that on $\C^5$, we have $g\neq 0$ exactly encodes the constraint that $i\neq 0, j\neq 0, i\neq j, r\neq 0$. 
In light of this, our goal then is to show that affine varieties induced by $I$ and $g$ are the same
\[
	\mathbf{V}(I)= \mathbf{V}(g),
\]
since this implies that the only common zeros of $\{f_1,\ldots,f_s\}$ in $\C^5$ are those with $i=0,j=0$, $i=j$ or $r=0$. By the strong Nullstellensatz  [Ch 4, Theorem 10 \cite{Cox}] (see also Theorem \ref{thm:SNull} below),  this is true if and only if there exists an $n\in \Z_{\geq 0}$ such that $g^n \in I$, or equivalently by Theorem \ref{thm:Saturation}, if the reduced Gr\"obner basis of the saturation $I:g^\infty$ for any given monomial ordering is $\{1\}$ (see Appendix \ref{sec:Sat} for background on saturation). 
To compute a saturation with respect to a single polynomial, it suffices to introduce an extra variable $z$ to represent $1/g$ and consider the augmented ideal 
\[
	\tilde{I} = \langle f_1,\dots,f_s, gz-1\rangle \subseteq \C[i_1,i_2,j_1,j_2,r,z].
\]
By Theorem \ref{thm:Saturation}, if $\{1\}$ is the reduced Gr\"obner basis for $\tilde{I}$, then it is also the reduced Gröbner basis for $I:g^\infty$.

We use Maple \cite{maple} to compute the reduced Gr\"obner basis $G$ for the ideal $\tilde{I}$. 
This computation is done in graded reverse lexicographical order (or ``grevlex'') and the variable ordering
\[
	i_1 < i_2 < j_1 < j_2 < z <r
\] 
using an implementation of the F4 algorithm \cite{F1999}; see Appendix \ref{sec:code}. 
The result is $G = \{1\}$, thereby concluding the proof.
\end{proof}
\section{Verifying distinctness in the diagonal algebra} \label{sec:distinct}
%!TEX root = master.tex

So far we have shown that if $\mathfrak{h}$ contains a {\em distinct matrix} $\mathbb{D}$ in the sense of Definition \ref{def:distinctless-reform} then $\mathrm{Lie}(\{H^k\}) = \mathfrak{sl}_{\Z^2_{0,N}}(\R)$, which implies projective spanning by Proposition \ref{prop:sufficient-projectspan-NS}.

The goal of this section is to show that $\mathfrak{h}$ does contain many distinct matrices, in fact, they are  `generic' in the sense that they form an open and dense set in $\mathfrak{h}$. 
Unfortunately, each individual diagonal matrix $\mathbb{D}^k = [H^k,H^{-k}]$ is certainly not distinct, since there are many degeneracies related to each particular $k$. However, we have the benefit of a large number of such diagonal matrices and can take linear combinations of each $\mathbb{D}^k$ to find a distinct matrix. Specifically taking linear combinations allows one to reduce the condition for the distinctness condition \eqref{eq:symm-distinct} to one that is much more mild on the entire collection $\{\mathbb{D}^k\}$.

Indeed, the main result of this section, and the main effort of proof is to show the following sufficient condition on the collection $\{\mathbb{D}^k\}$.
\begin{proposition}\label{prop:main-distinctness-prop}
For each $(i,j,\ell,m)\in\mathcal{C}_N$ (defined in \eqref{eq:constraint-set}) with $i+j+\ell+m=0$, there exists a $k\in\Z^2_{0,N}$ such that 
\[
	\mathbb{D}^k_i + \mathbb{D}^k_j + \mathbb{D}_\ell^k + \mathbb{D}^k_m \neq 0.
\]
\end{proposition}

We now show how proposition \ref{prop:main-distinctness-prop} implies that ``most'' elements in the span of $\{\mathbb{D}^k\}$ are in fact distinct.

\begin{lemma}\label{lem:linear-combin-regular}
Assume the result of Proposition \ref{prop:main-distinctness-prop} holds, then there exists an open and dense set of matrices in $\mathfrak{h} = \Span\{\mathbb{D}^k\,:\, k\in\Z^2_{0,N}\}$ that are distinct in the sense of definition \ref{def:distinctless-reform}.
\end{lemma}
\begin{proof}
For each fixed $(i,j,\ell,m)\in \mathcal{C}_N$, we denote the vector
\[
	w_{(i,j,\ell,m)} := \left(\mathbb{D}_i^k + \mathbb{D}_j^k + \mathbb{D}_{\ell}^k + \mathbb{D}_m^k\,:\, k\in \Z^2_{0,N}\right) \in \R^{\Z^2_{0,N}},
\]
and for each $(i,j,\ell,m)\in \mathcal{C}_N$, let
\[
	\Gamma_{(i,j,\ell,m)} = \left\{\alpha\in\R^{\Z^2_{0,N}}\,:\, \alpha\cdot w_{(i,j,\ell,m)} \neq 0 \right\}.
\]
By Proposition \ref{prop:main-distinctness-prop}, $w_{(i,j,\ell,m)}$ is always a non-zero vector and hence $\Gamma_{(i,j,\ell,m)}$ is an open-dense set (being the complement of a plane). Since $\mathcal{C}_N$ is a finite set
\[
	\Gamma = \bigcap_{(i,j,\ell,m)\in \mathcal{C}_N}\Gamma_{(i,j,\ell,m)}
\]
is also open and dense in $\R^{\Z^2_{0,N}}$. It follows that for any $\alpha \in \Gamma$, the linear combination $\sum_{k} \alpha_k \mathbb{D}^k$, is distinct.
\end{proof}

We now briefly summarize how Proposition \ref{prop:main-distinctness-prop} completes the proof of the main results of our paper. 

\begin{proof}[\textbf{Proof of Theorems \ref{thm:Main-projective-spanning} and \ref{thm:mainchaos}}]
	Proposition \ref{prop:main-distinctness-prop} together with Lemma \ref{lem:linear-combin-regular} imply that $\mathfrak{h}$ contains a distinct matrix in the sense of definition \ref{def:distinctless-reform}. 
	Then, Corollary \ref{cor:disDon} and Proposition \ref{prop:sufficient-projectspan-NS} implies projective hypoellipticity for the Naver-Stokes equations, i.e. Theorem \ref{thm:Main-projective-spanning}, from which Theorem \ref{thm:mainchaos} follows by the results of [Theorem C; \cite{BBPS20}]; see Sections \ref{sec:FI} and \ref{sec:GNSEproSpan} for more information. 
\end{proof}

\subsection{The simpler ``infinite dimensional'' case}\label{subsec:simple-case}

Before continuing with the proof of proposition \ref{prop:main-distinctness-prop}, which is rather technical in nature due to the presence of the Galerkin truncation, it is very instructive to first see how the proof goes in the ``infinite dimensional'' case when the Galerkin truncation is removed and we instead consider the entire lattice $\Z^2_0$. The actual proof is similar in spirit to the one presented below, just repeated 35 times to cover various edge cases. The proof in this section has an accompanying Maple worksheet that will do the algebraic computations and compute the reduced Gr\"obner bases using exact arithmetic; see Appendix \ref{subsec:alg-geo-proof-2}. 

The proof full proof of Proposition \ref{prop:main-distinctness-prop} will make use of the algebraic structure of $\mathbb{D}(i,k,r) = \mathbb{D}_i^k$ (recall $r$ dependence is implicit) as a piecewise defined rational function on $(\Z^2_{0,N})^2$. The overall goal is to show that for each side length $r\neq 0$ that there do not exist any solutions $(i,j,\ell,m)\in \mathcal{C}_N$ with $i+j+\ell+m=0$, to the set of Diophantine equations
\begin{equation}\label{eq:Diophantine}
	\mathbb{D}^k_i + \mathbb{D}^k_j + \mathbb{D}_\ell^k + \mathbb{D}^k_m = 0, \quad \text{for all }k\in\Z^2_{0,N}.
\end{equation}

 As mentioned in Remark \ref{rem:infinite-algebra-diag} without the Galerkin cut-off $\mathbb{D}(i,k,r)$ takes a much simpler rational algebraic form that is not piecewise defined on the lattice,
\begin{equation}\label{eq:D-inf-def}
	\bar{\mathbb{D}}(i,k,r) = c_{i,k}c_{i+k,k} - c_{i,k}c_{i-k,k} = \langle i^\perp, k\rangle^2_r\left(\frac{1}{|k|_r^2} - \frac{1}{|i|_r^2}\right)\left(\frac{1}{|i-k|_r^2} - \frac{1}{|i+k|_r^2}\right).
\end{equation}
In light of this, our strategy is to extend the rational function
\[
	\overline{\mathbb{W}}(i,j,\ell,m,k,r):= \bar{\mathbb{D}}(i,k,r) + \bar{\mathbb{D}}(j,k,r) + \bar{\mathbb{D}}(\ell,k,r) + \bar{\mathbb{D}}(m,k,r)
\]
in 11 variables
\[
i = (i_1,i_2),\, j=(j_1,j_2),\, m = (m_1,m_2),\, \ell = (\ell_1,\ell_2),\, k = (k_1,k_2), \quad \text{and}\quad r,
\]
to the algebraically closed field $\C$, and show that such a system of algebraic equations is inconsistent. In particular, the numerator polynomial
\[
	P(i,j,\ell,m,k,r) = \mathrm{numer}\left(\overline{\mathbb{W}}(i,j,\ell,m,k,r)\right)
\]
belongs to $\C[i,j,\ell,m,k,r]$\footnote{Here we use the obvious shorthand $\C[i,j,\ell,m,k,r] = \C[i_1,i_2,j_1,j_2,\ell_1,\ell_2,m_1,m_2,k_1,k_2,r]$.}, has integer coefficients and vanishes whenever $\overline{\mathbb{W}}$ does. In many ways it is this polynomial and the fact that it has finite order and integer coefficients that allows for the use of a computer algebra proof that holds for arbitrary Galerkin truncation. 

The goal is to understand the common zeros of the collection of polynomials obtained by evaluating $k$ on various subsets of the lattice. Particularly for a given subset $K \subseteq \Z^2_{0}$, we consider the ideal of polynomials in the remaining $9$ variables $(i,j,\ell,m,r)\in \C^9$ generated by evaluating $P$ at each $k\in K$
\[
	I_{K} := \left\langle P_k\,:\,k\in K \right\rangle\subset\C[i,j,\ell,m,r], \quad \text{where}\quad P_k(i,j,\ell,m,r) = P(i,j,\ell,m,k,r).
\]
We also introduce the two polynomials $h_1,h_2$ describing the $i+j+\ell +m =0$ constraint
\[
h_1 := i_1 + j_1 + \ell_1 + m_1\quad \text{and}\quad h_2 := i_2 + j_2 + \ell_2 + m_2
\]
as well as the polynomial\footnote{Note the choice of $|\cdot|_r$ norm here, which is fundamental to ensuring our computational implementation converges for arbitrary $r$.}
\begin{equation}\label{eq:g-poly}
	g(i,j,\ell,m,r) := r^2|i|^2_r|j|^2_r|\ell|^2_r|m|^2_r(|i+j|^2_r + |\ell + m|^2_r)(|i+\ell|^2_r + |j+m|^2_r)(|i+m|^2_r+|j+\ell|^2_r)
\end{equation}
whose non-vanishing implies that 
\[
	r\neq 0, i\neq 0, j\neq 0,\ell \neq 0, m\neq 0
\]
and
\[
	(i+j,\ell +m) \neq (0,0),\, (i+ \ell,j+m)\neq (0,0),\, (i+m,j + \ell) \neq (0,0),
\]
and therefore $\{g\neq 0\}$ perfectly encodes the constraint set $\mathcal{C}_N$ along with the assumption that $r\neq 0$. The proof will be complete then, if we can find a set $K \subseteq \Z^2_{0}$ so that the affine variety generated by $I_{K},h_1,h_2$ is the same as that generated by $g$, namely
\[
	\mathbf{V}(I_{K},h_1,h_2) = \mathbf{V}(g).
\]
In order to show this, it will also be useful to freeze $i,j,\ell,m,r$ and treat $P$ as a polynomial in $k=(k_1,k_2)$ and regard the coefficients as polynomials in $\mathbb{C}[i,j,\ell,m,r]$. We denote the collection of these polynomials by
\[
	\{f_1,\ldots, f_s\} := \mathrm{coeffs}(P,\{k_1,k_2\}) \subseteq \mathbb{C}[i,j,\ell,m,r].
\]
Note that the number of polynomials in $\{f_1,\ldots, f_s\}$ only depends on the order of the polynomial $P$ in $k$ and is independent of any truncation.

By Lemma \ref{lem:Reducek} and the observation that $P$ is order $19$ in $k$, if there $\exists k' \in \Z^2_0$ such that $\set{k \in \Z^2_0 :\abs{k-k'}_{\ell^\infty} \leq 10} \subseteq K$, then the polynomial ideals are equal
\begin{align*}
I_K = \brak{f_1,...,f_s}.
\end{align*}
Since we can obviously find such a set $K$,  our goal reduces to showing that
\[
	\mathbf{V}(f_1,\ldots,f_s,h_1,h_2) = \mathbf{V}(g).
\]
By Theorem \ref{thm:SNull}, this is equivalent to showing that the saturated ideal $\langle f_1,\ldots,f_s,h_1,h_2\rangle:g^{\infty}$ satisfies
\[
\langle f_1,\ldots,f_s,h_1,h_2\rangle:g^{\infty} = \C[i,j,\ell,m,r],
\]
or more practically from a computational stand point, by Theorem \ref{thm:Saturation} that $\{1\}$ is the reduced Gr\"obner basis for the augmented ideal
\[
	\tilde{I} = \langle f_1,\ldots,f_s,h_1,h_2,zg-1\rangle \subseteq \C[i,j,\ell,m,r,z]
\]
with $z$ added as an extra variable. This can indeed be checked computationally in Maple. Specifically, we compute the reduced Gr\"obner basis of $\tilde{I}$ using the graded reverse lexicographical monomial order (or “grevlex”) and the variable ordering
\[
i_1<i_2<j_1<j_2<\ell_1<\ell_2<m_1<m_2<z < r
\]
using an implementation of the F4 algorithm \cite{F1999} (see Appendix \ref{sec:Groebner} and \ref{sec:code} ); the computation verifies that $G = \{1\}$, thereby concluding the proof.

\subsection{Treating the Galerkin truncation: Proof of proposition \ref{prop:main-distinctness-prop}}\label{subsec:proof-of-distinctness}

As already mentioned, the Galerkin truncation makes $\mathbb{D}(i,k,r) = \mathbb{D}_i^k$ a piecewise defined rational function (depending on the choice of $k$) and so great care must be taken to consider all the possible combinations algebraic forms on different partitions of the lattice to carry out a similar argument to the one given above.

\begin{proof}[Proof of proposition \ref{prop:main-distinctness-prop}] To begin, it is convenient to write $\mathbb{D}^k_i$ in a proper piecewise defined sense. To do this we will find it convenient to define the set
\[
	\mathcal{S}^k := \Z^2_{0,N}\cap\{\Z^2_{0,N} + k\} \subseteq \Z^2_{0,N},
\]
and denote
\[
	\mathbb{D}_+(i,k,r) := c_{i,k}c_{i+k,k} = \langle i^\perp,k\rangle_r^2\left(\frac{1}{|k|_r^2} - \frac{1}{|i|_r^2}\right)\left(\frac{1}{|k|_r^2}-\frac{1}{|i+k|^2_r}\right)
\]
as well as
\[
	\mathbb{D}_-(i,k,r) := - \mathbb{D}_+(i,-k,r) = \langle i^\perp,k\rangle_r^2\left(\frac{1}{|k|_r^2} - \frac{1}{|i|_r^2}\right)\left(\frac{1}{|i-k|_r^2}-\frac{1}{|k|^2_r}\right)
	\]
and
\[
	\bar{\mathbb{D}}(i,k,r) := \mathbb{D}_+(i,k,r) + \mathbb{D}_-(i,k,r) = \langle i^\perp, k\rangle^2_r\left(\frac{1}{|k|_r^2} - \frac{1}{|i|_r^2}\right)\left(\frac{1}{|i-k|_r^2} - \frac{1}{|i+k|_r^2}\right).
\]
Then $\mathbb{D}(i,k,r)$ can be written piecewise as
\[
	\mathbb{D}(i,k,r) = \begin{cases}
	\mathbb{D}_+(i,k,r) & i \in \mathcal{S}^k\backslash \mathcal{S}^{-k}\\
	\bar{\mathbb{D}}(i,k,r) & i\in \mathcal{S}^k\cap\mathcal{S}^{-k}\\
	\mathbb{D}_-(i,k,r) & i \in \mathcal{S}^{-k}\backslash \mathcal{S}^{k}\\
	0 & i\in \Z^{2}_{0,N}\backslash(\mathcal{S}^k\cup\mathcal{S}^{-k})
	\end{cases}.
\]
This means that for each fixed $i\in \Z^{2}_{0,N}$, $\mathbb{D}(i,k,r)$ is obtained by evaluating exactly one of the four exact rational algebraic functions belonging to $\{\mathbb{D}_+,\bar{\mathbb{D}},\mathbb{D}_-,0\}$, moreover it is not hard to show that for each such $i\in \Z^2_{0,N}$ there is always a suitably large set $K$ such that $\mathbb{D}(i,k,r)$ takes the same algebraic form for all $k\in K$. Based off this idea, it is the following key lemma that allows us to treat the piecewise rational behavior of the sum 
\[
	\mathbb{W}(i,j,\ell,m,k,r):= \mathbb{D}(i,k,r) + \mathbb{D}(j,k,r) + \mathbb{D}(\ell,k,r) + \mathbb{D}(m,k,r)
\] in a purely algebraic fashion, as long as $N$ is large enough.
\begin{lemma}\label{lem:key-lemma}
Let $a\geq 1$ and suppose that $N > 4(9a + 8)$, then for each fixed $(i,j,\ell,m)\in \mathcal{C}_N$, there exists a $k^\prime\in \Z^2_{0,N-a}$ and four rational functions $\{\mathbb{D}_1,\mathbb{D}_2,\mathbb{D}_3,\mathbb{D}_4\}$,
each taking one of the four possible forms $\{\mathbb{D}_+,\bar{\mathbb{D}},\mathbb{D}_-,0\}$, with at least one the $\mathbb{D}_i \neq 0$ such that for all $k \in \Z^2_0$ with $|k - k^\prime|_{\ell^\infty} \leq a$, $\mathbb{W}$ takes the form
\[
	\mathbb{W}(i,j,\ell,m,k,r) = \mathbb{D}_1(i,k,r) + \mathbb{D}_2(j,k,r) + \mathbb{D}_3(\ell,k,r) + \mathbb{D}_4(m,k,r).
\]
\end{lemma}

Before proving Lemma \ref{lem:key-lemma} (which is done in the following subsection), lets see how to use it to complete the proof of Proposition \ref{prop:main-distinctness-prop}.
Indeed, with this Lemma in hand, the proof now follows along the similar lines as that of Section \ref{subsec:simple-case} above. Fix $(i,j,\ell,m)\in \mathcal{C}_N$ and $r\neq 0$ with $i+j+\ell+ m =0$ and assume by contradiction that
\[
	\mathbb{W}(i,j,\ell,m,k,r) = 0 \quad \text{for all}\quad k \in \Z^2_{0,N}.
\]
By Lemma \ref{lem:key-lemma}, let
\[
	K := \{k\in\Z^2_{0,N}\,:\,|k-k^\prime|_{\ell^\infty} \leq a\}
\] 
so that for all $k\in K$, $\mathbb{W}(i,j,\ell,m,k,r)$ is a given fixed rational function and therefore we can define as in Section \ref{subsec:simple-case} the numerator polynomial
\[
	P(i,j,\ell,m,k,r) = \mathrm{numer}\left(\mathbb{W}(i,j,\ell,m,k,r)\right),
\]
and observe that by assumption we have
\[
	(i,j,\ell,m,r) \in \mathbf{V}(I_K,h_1,h_2), \quad \text{where}\quad I_K := \langle P_k\,:\, k\in K\rangle.
\]
Therefore the proof is complete if we can show that $\mathbf{V}(I_{K},h_1,h_2) = \mathbf{V}(g)$, where $g$ is defined by \eqref{eq:g-poly}. 
By Lemma \ref{lem:Reducek} and Lemma \ref{lem:key-lemma}, using that the $P$ is at most order 19 in $k$, we can choose $a =10$ so that $2a > 19$, and therefore the polynomials $\{f_1,\ldots,f_s\} = \mathrm{coeffs}(P,\{k_1,k_2\})$ generate the ideal $I_{K}$. We can now follow the exact same procedure as in Section \ref{subsec:simple-case} to show that the reduced Gr\"obner basis of the extended ideal $\tilde{I} = \langle f_1,\ldots,f_s,h_1,h_2,zg-1\rangle$ is $\{1\}$. This can be verified computationally for general algebraic forms
\begin{equation}\label{eq:symmetry-W}
	\mathbb{W}(i,j,\ell,m,k,r)= \mathbb{D}_1(i,k,r) + \mathbb{D}_2(j,k,r) + \mathbb{D}_3(\ell,k,r) + \mathbb{D}_4(m,k,r)
\end{equation}
by computing the Gr\"obner basis of $\tilde{I}$ for all possible choices of algebraic functions $(\mathbb{D}_1,\mathbb{D}_2,\mathbb{D}_3,\mathbb{D}_4)$ sampled from the set $\{\mathbb{D}_+,\bar{\mathbb{D}},\mathbb{D}_-,0\}$ (excluding $(0,0,0,0)$). Due to symmetry of the sum \eqref{eq:symmetry-W} and the constraint $i+j+\ell+m =0$, up to relabeling $i,j,\ell,m$ we can disregard the order in which we consider $\{\mathbb{D}_1,\mathbb{D}_2,\mathbb{D}_3,\mathbb{D}_4\}$ and therefore the total number of possible ideals we have to check is just the number of ways to draw an unordered sample $\{\mathbb{D}_1,\mathbb{D}_2,\mathbb{D}_3,\mathbb{D}_4\}$ of $4$ things from a the set $\{\mathbb{D}_+,\bar{\mathbb{D}},\mathbb{D}_-,0\}$ with replacement (excluding $\{0,0,0,0\}$). This means that there are
\[
	{ 4 + 4-1\choose 4}-1 = { 7\choose 4}-1 = 34
\]
different ideals $\tilde{I}$ to check. This seemingly tedious task is carried effortlessly by Maple (see Appendix \ref{subsec:alg-geo-proof-3} for the code) showing that the reduced Gr\"obner basis for $\tilde{I}$ in every case is $\{1\}$. This concludes the proof of Proposition \ref{prop:main-distinctness-prop} (and hence of Theorem \ref{thm:mainchaos}). 
\end{proof}

\subsubsection{Proof of Lemma \ref{lem:key-lemma}}
We now turn to the proof of the key Lemma \ref{lem:key-lemma}.

\begin{proof}[Proof of Lemma \ref{lem:key-lemma}]

Fix an integer $a> 0$ and for each $k^\prime\in\Z^{2}_{0,N-a}$, let $\mathcal{K}^{k^\prime} := \{k\in\Z^{2}_{0,N}\,:\, |k-k^\prime|_{\ell^{\infty}}\leq a\}$ and define the following sets
\[
	\mathcal{G}^{k^\prime}_+ := \bigcap_{k \in \mathcal{K}^{k^\prime}}\mathcal{S}^{k}\backslash \mathcal{S}^{-k},\quad
	\mathcal{G}^{k^\prime}_- := \bigcap_{k \in \mathcal{K}^{k^\prime}}\mathcal{S}^{-k}\backslash \mathcal{S}^{k},\quad
	\bar{\mathcal{G}}^{k^\prime} := \bigcap_{k \in\mathcal{K}^{k^\prime}}\mathcal{S}^{k}\cap \mathcal{S}^{-k},\quad \mathcal{G}^{k^\prime}_0 := \bigcap_{k \in\mathcal{K}^{k^\prime}}\Z^{2}_{0,N}\backslash(\mathcal{S}^k\cup\mathcal{S}^{-k}),
\]
where $\mathbb{D}(i,k,r)$ takes a specific algebraic form ($\{\mathbb{D}_+,\mathbb{D}_-,\bar{\mathbb{D}},0\}$ respectively) uniformly for $k$ satisfying $|k-k^\prime|_{\ell^\infty}\leq a$. This means we have a ``good'' set
\[
\mathcal{G}^{k^\prime} := \mathcal{G}^{k^\prime}_+\cup \mathcal{G}^{k^\prime}_-\cup \bar{\mathcal{G}}^{k^\prime}\cup \mathcal{G}^{k^\prime}_0
\]
where for each fixed $i\in \mathcal{G}^{k^\prime}$, $\mathbb{D}(i,k,r)$ takes a consistent algebraic form for all $|k-k^\prime|_{\ell^\infty}\leq a$, and a remaining ``bad'' set 
\[
\mathcal{B}^{k^\prime} := \Z^2_{0,N}\backslash\mathcal{G}^{k^\prime}
\]
 where $\mathbb{D}(i,k,r)$ doesn't take a consistent algebraic form for all $|k-k^\prime|_{\ell^\infty}\leq a$ (see Figure \ref{fig:good-bad-scheme} for a illustration of these sets in the lattice).

\begin{figure}[H]
\centering
\def\svgwidth{.6\textwidth}
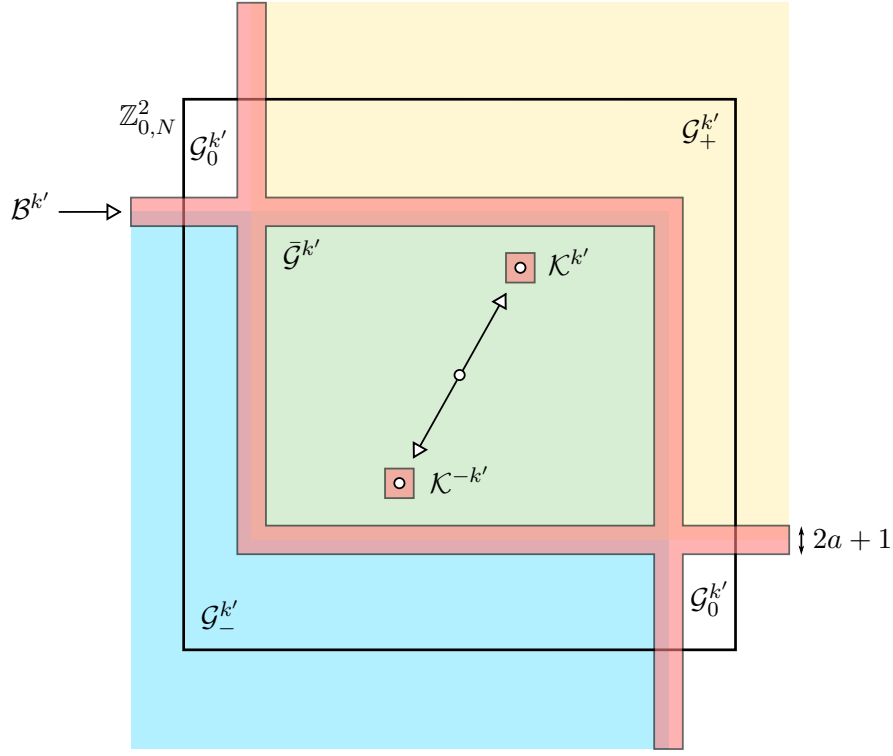
\caption{A schematic illustration for the partition of $\Z^2_{0,N}$ into good $\mathcal{G}^{k^\prime} := \mathcal{G}^{k^\prime}_+\cup \mathcal{G}^{k^\prime}_-\cup \bar{\mathcal{G}}^{k^\prime}\cup \mathcal{G}^{k^\prime}_0$ and bad sets $\mathcal{B}^{k^\prime}$ for a given $k^\prime$. The set $\mathcal{G}^{k'}_+$ is the upper right L-shaped zone (shaded yellow), the set $\mathcal{G}^{k'}_-$ is the lower left L-shaped zone (shaded blue), and the set $\bar{\mathcal{G}}^{k'}$ is the center square with two punctures (shaded green). The upper left and lower right corners (not shaded) are the two pieces of $\mathcal{G}^{k'}_0$. The bad set consists of several pieces (all shaded red): the the connection zones between all of the $\mathcal{G}$ sets as well as the squares $\mathcal{K}^{\pm k'}$ shown here embedded in $\bar{\mathcal{G}}^{k'}$.} \label{fig:good-bad-scheme}
\end{figure}

For a given $\{i,j,\ell,m\}\in \Z^{2}_{0,N}$ our goal is to find a  $k^\prime$ such that $\{i,j,\ell,m\}\subseteq \mathcal{G}^{k^\prime}$, such that not all $\{i,j,\ell,m\}$ belong to $\mathcal{G}^{k^\prime}_0$. This is the content of the following Lemma.
\begin{lemma}\label{lem:inclusion-exclusion}
Suppose $N > 4(9a + 8)$, then for every $\{i,j,\ell,m\}\subseteq \Z^2_{0,N}$ there exists a $k^\prime\in \Z^2_{0,N}$ with $|k^\prime|_{\ell^\infty}\leq \lfloor N/2\rfloor-a$ such that $\{i,j,\ell,m\}\subseteq \mathcal{G}^{k^\prime}$.
\end{lemma}
\begin{proof}
Some of this is best seen by picture. First, we note that the requirement that $|k^\prime|_{\ell^\infty}\leq \lfloor N/2\rfloor-a$ is solely so that the sets $\mathcal{K}^{\pm k^\prime}$ stay with in the inner-most rectangle (this is to avoid keeping track of much more complicated intersections between bad sets).  For simplicity we denote for each $n\geq 1$ the diagonal element $k^\prime_n = \left(a + 2(n-1)(a+1),a + 2(n-1)(a+1)\right)$ and let $\mathcal{B}_n := \mathcal{B}^{k^\prime_n}$. We note that each $1\leq n_1< n_2$, $\mathcal{B}_{n_1}$ always has has a non-trivial overlap with $\mathcal{B}_{n_2}$, 
\[
\mathcal{B}_{n_1}\cap\mathcal{B}_{n_2} \neq \emptyset, \quad \text{if}\quad 1\leq n_1 < n_2.
\]
However, if $N$ is big enough, any triple intersection is empty (see Figure \ref{fig:bad-sets} for a depiction of this)
\[
	\mathcal{B}_{n_1}\cap\mathcal{B}_{n_2}\cap \mathcal{B}_{n_3} = \emptyset, \quad \text{if}\quad 1\leq n_1<n_2<n_3.
\]
By ``big enough'', we mean that $\max_i |k^\prime_{n_i}|_{\ell^\infty} \leq \lfloor N/2\rfloor-a$ so that sets $\mathcal{K}^{k^\prime_{n_i}}\cup \mathcal{K}^{-k^\prime_{n_i}}$ are all disjoint and don't overlap with any of the outer bands.  

\begin{figure}[H]
\centering
\def\svgwidth{.6\textwidth}
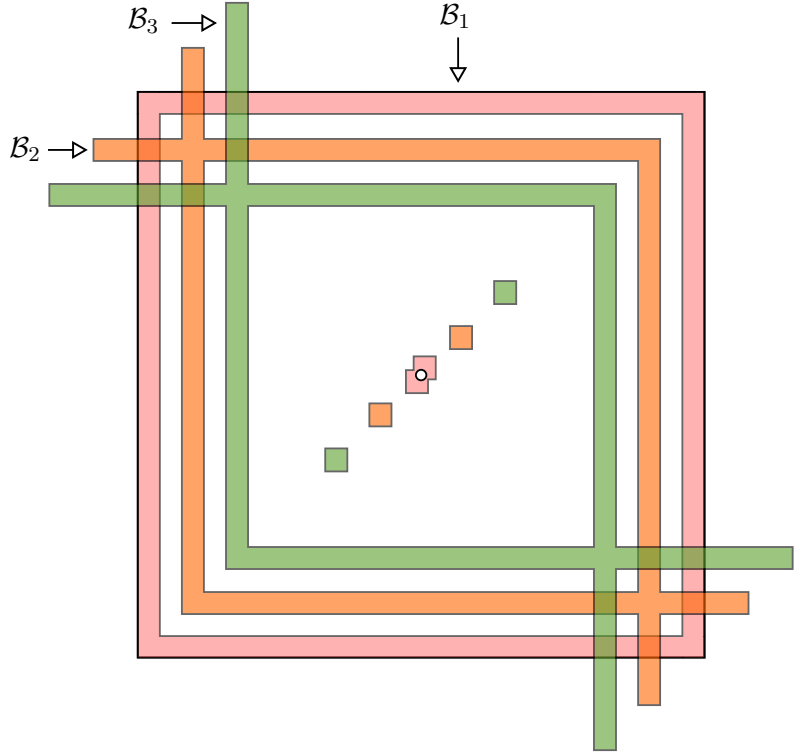
\caption{An illustration of the empty triple intersection $\mathcal{B}_{1}\cap \mathcal{B}_2\cap \mathcal{B}_3 = \emptyset$. The gaps between sets are exaggerated for visual clarity.}
\label{fig:bad-sets}
\end{figure}

 For each $i\in \Z^{2}_{0,N}$, denote $\delta_i$ the delta measure on $\Z^2_{0,N}$ concentrated at $i$, defined for each $A\subseteq \Z^2_{0,N}$ by
\[
	\delta_i(A) = \begin{cases}
	1 & i\in A\\
	0 & i \notin A
	\end{cases}.
\]
Likewise for any four lattice points $\{i,j,\ell,m\}\subseteq \Z^2_{0,N}$, denote the  counting measure
\[
	\gamma_{i,j,\ell,m} := \delta_i + \delta_j + \delta_\ell + \delta_m,
\]
which counts how many of the lattice points $\{i,j,\ell,m\}$ belong to a given subset of the lattice. Note that for any $A\subseteq \Z^2_{0,N}$, $0\leq \gamma_{i,j,\ell,m}(A) \leq 4$.

We now work by contradiction and assume that there exists four lattice points $\{i,j,\ell,m\}\in \Z^{2}_{0,N}$ such that for every $k^\prime\in \Z^{2}_{0,N}$ with $|k^\prime|_{\ell^\infty}\leq \lfloor N/2\rfloor-a$, at least one of the lattice points $\{i,j,\ell,m\}$ belongs to $\mathcal{B}^{k^\prime}$. This implies that for each $n\geq 1$ we have
\[
	\gamma_{i,j,\ell,m}(\mathcal{B}_n) \geq 1.
\]
By the inclusion-exclusion principle (and the fact that the only non-trivial intersections are pairwise intersections), we have that for each $\{i,j,\ell,m\}\subseteq \Z^2_{0,N}$ and $M > 1$
\begin{equation}
\begin{aligned}\label{eq:inclusion-exclusion}
\gamma_{i,j,\ell,m}\left(\bigcup_{1\leq n \leq M} \mathcal{B}_n\right) &= \sum_{n=1}^{M} \gamma_{i,j,\ell,m}\left(\mathcal{B}_n\right) -\sum_{1\leq n_1 < n_2 \leq M}\gamma_{i,j,\ell,m}(\mathcal{B}_{n_1}\cap\mathcal{B}_{n_2})\\
 &= \sum_{n=1}^{M} \gamma_{i,j,\ell,m}\left(\mathcal{B}_n\right) - \gamma_{i,j,\ell,m}\left(\bigcup_{1\leq n_1 < n_2\leq M}\mathcal{B}_{n_1}\cap\mathcal{B}_{n_2}\right)\\
 & \geq M - 4.
\end{aligned}
\end{equation}
Choosing $M = 9$, then implies that
\[
	\gamma_{i,j,\ell,m}\left(\bigcup_{1\leq n \leq M} \mathcal{B}_n\right) \geq 5
\]
which is clearly a contradiction, since $\gamma_{i,j,\ell,m}\leq 4$. Since we had to take $M = 9$, this means that we need to have 
\[
	2(|k^\prime_{9}|_{\ell^\infty} + a) < N,
\]
which is the same as requiring that $N > 4(9a + 8)$.
\end{proof}

To complete the proof of Lemma \ref{lem:key-lemma}, we may assume with out loss of generality that not all $\{i,j,\ell,m\}$ belong to $\mathcal{G}_0^{k^\prime}$, since if that were the case we could replace $k^\prime$ with its horizontal reflection $\hat{k}^\prime = (-k_1^\prime,k_2^\prime)$, and obtain $\{i, j,\ell, m\}\in \mathcal{G}^{\hat{k}^\prime}_+\cap\mathcal{G}^{\hat{k}^\prime}_-\subseteq \mathcal{G}^{\hat{k}^\prime}$ (see Figure \ref{fig:reflection} for a visual proof of this). Then it is clear that Lemma \ref{lem:inclusion-exclusion} implies that there are rational functions $\{\mathbb{D}_1,\mathbb{D}_2,\mathbb{D}_3,\mathbb{D}_4\}$ with each $\mathbb{D}_i$ belonging to $\{\mathbb{D}_+,\mathbb{D}_-,\bar{\mathbb{D}},0\}$ (excluding the case where they are all zero) such that
\[
	\mathbb{W}(i,j,\ell,m,k,r) = \mathbb{D}_1(i,k,r) + \mathbb{D}_2(j,k,r) + \mathbb{D}_3(\ell,k,r) + \mathbb{D}_4(m,k,r).
\]
\begin{figure}[H]
\centering
\def\svgwidth{.9\textwidth}
%% Creator: Inkscape inkscape 0.92.5, www.inkscape.org
%% Accompanies image file 'drawing-2.pdf' (pdf, eps, ps)
\begingroup%
  \makeatletter%
  \providecommand\color[2][]{%
    \errmessage{(Inkscape) Color is used for the text in Inkscape, but the package 'color.sty' is not loaded}%
    \renewcommand\color[2][]{}%
  }%
  \providecommand\transparent[1]{%
    \errmessage{(Inkscape) Transparency is used (non-zero) for the text in Inkscape, but the package 'transparent.sty' is not loaded}%
    \renewcommand\transparent[1]{}%
  }%
  \providecommand\rotatebox[2]{#2}%
  \newcommand*\fsize{\dimexpr\f@size pt\relax}%
  \newcommand*\lineheight[1]{\fontsize{\fsize}{#1\fsize}\selectfont}%
  \ifx\svgwidth\undefined%
    \setlength{\unitlength}{693.34933063bp}%
    \ifx\svgscale\undefined%
      \relax%
    \else%
      \setlength{\unitlength}{\unitlength * \real{\svgscale}}%
    \fi%
  \else%
    \setlength{\unitlength}{\svgwidth}%
  \fi%
  \global\let\svgwidth\undefined%
  \global\let\svgscale\undefined%
  \makeatother%
  \begin{picture}(1,0.4527854)%
    \lineheight{1}%
    \setlength\tabcolsep{0pt}%
    \put(0,0){\includegraphics[width=\unitlength,page=1]{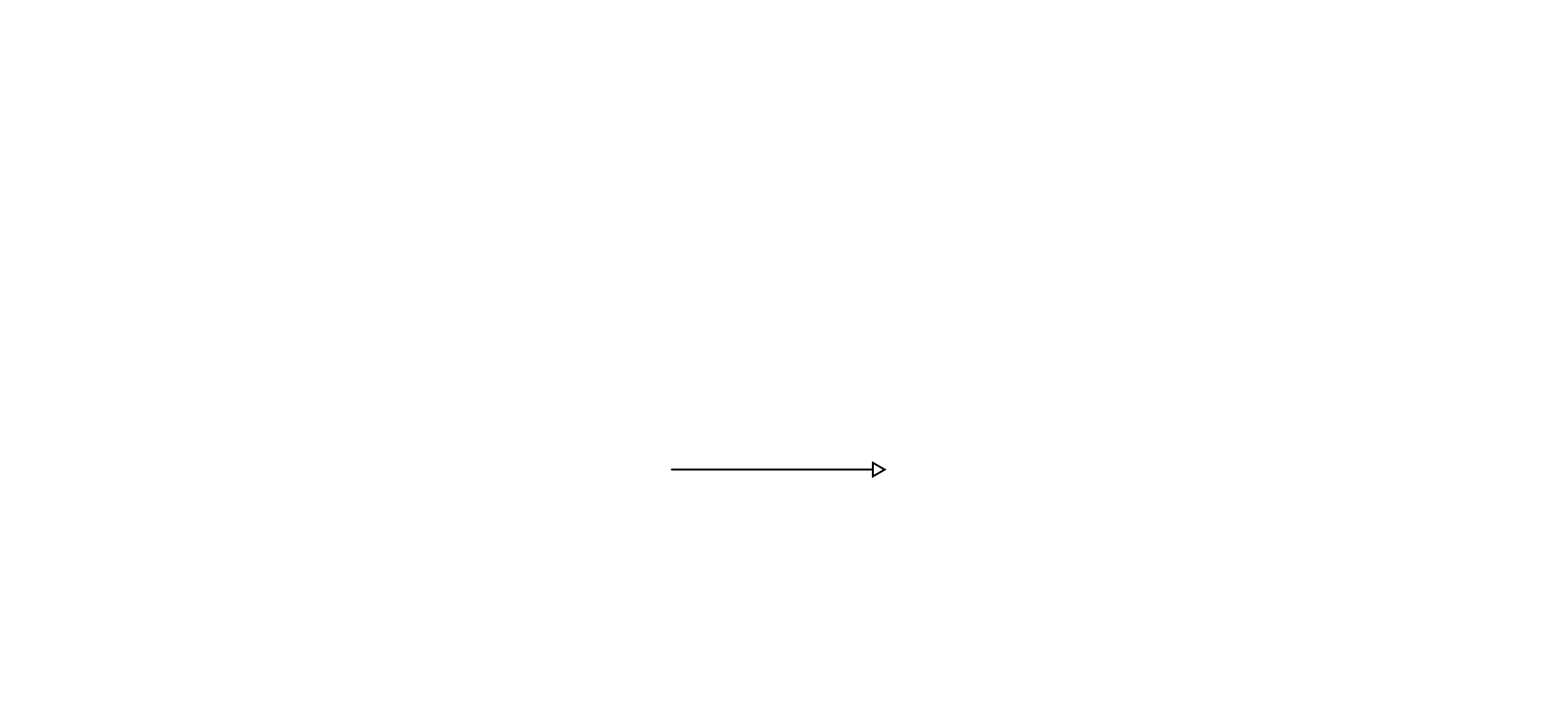}}%
    \put(0.4539096,0.16887685){\color[rgb]{0,0,0}\makebox(0,0)[lt]{\lineheight{1.25}\smash{\begin{tabular}[t]{l}$k^\prime\to \hat{k}^\prime$\end{tabular}}}}%
    \put(0,0){\includegraphics[width=\unitlength,page=2]{drawing-2.pdf}}%
    \put(0.1812614,0.15560648){\color[rgb]{0,0,0}\makebox(0,0)[lt]{\lineheight{1.25}\smash{\begin{tabular}[t]{l}$-k^\prime$\end{tabular}}}}%
    \put(0,0){\includegraphics[width=\unitlength,page=3]{drawing-2.pdf}}%
    \put(0.25469894,0.2883157){\color[rgb]{0,0,0}\makebox(0,0)[lt]{\lineheight{1.25}\smash{\begin{tabular}[t]{l}$k^\prime$\end{tabular}}}}%
    \put(0,0){\includegraphics[width=\unitlength,page=4]{drawing-2.pdf}}%
    \put(0.78073888,0.28353772){\color[rgb]{0,0,0}\makebox(0,0)[lt]{\lineheight{1.25}\smash{\begin{tabular}[t]{l}$\hat{k}^\prime$\end{tabular}}}}%
    \put(0.85429388,0.15234234){\color[rgb]{0,0,0}\makebox(0,0)[lt]{\lineheight{1.25}\smash{\begin{tabular}[t]{l}$-\hat{k}^\prime$\end{tabular}}}}%
  \end{picture}%
\endgroup%

\caption{An illustration of $\mathcal{G}^{k^\prime}_0 \subseteq \mathcal{G}^{\hat{k}^\prime}_+\cup\mathcal{G}^{\hat{k}^\prime}_-$. We see that reflection $k' = (k_1',k_2') \mapsto (-k_1',k_2') = \hat{k}'$ reverses upper and lower corners from right to left. Note we have chosen $k'$ sufficiently large for this and also to ensure that $\mathcal{K}^{\pm k'}$ and $\mathcal{K}^{\pm \hat{k}'}$ stay surrounded entirely by $\bar{\mathcal{G}}^{k'}$ and $\bar{\mathcal{G}}^{\hat{k}'}$ respectively.}
\label{fig:reflection}
\end{figure}
\end{proof}

\appendix

\section{Relevant algebraic geometry}\label{sec:Groebner}
%!TEX root = master.tex

\subsection{Polynomial ideals and Gr\"obner bases}

In this section we review some of the basic concepts from algebraic geometry that are used in the computer assisted proof of condition (i). 
We give a brief summary here for the readers' convenience as the area may be far removed from many of the readers' expertise.
The exposition here is adapted from Cox-Little-O'Shea \cite{Cox}, see therein for mathematical details and more explanations. 

Given a field $\K$ we denote by $\K[x_1,...,x_n]$ the ring of polynomials over $\K$ in $n$ variables with coefficients in $\K$. First, we recall the notion of a polynomial ideal.
\begin{definition}
A subset $I \subset \K[x_1,...,x_n]$ is an \emph{ideal} if 
\begin{itemize}
	\item[(i)] $0 \in I$. 
	\item[(ii)] If $f,g \in I$ then $f+g \in I$. 
	\item[(iii)] If $f \in I$ and $h \in \K[x_1,...,x_n]$, then $hf \in I$. 
\end{itemize}
Given $f_1,...,f_s$ polynomials in $\K[x_1,...,x_n]$, we define the ideal generated by these polynomials, denoted $\brak{f_1,...,f_s}$ as 
\begin{align*}
\brak{f_1,...,f_s} = \set{ \sum_{j=1}^s h_j f_j \,:\, h_j \in \K[x_1,...,x_n]}. 
\end{align*}
\end{definition}

We recall the notion of a variety, which is the set of points in $\K^n$ which solve a given system of polynomials. 
\begin{definition}
Let $f_1,\ldots,f_s \in \K[x_1,\ldots,x_n]$. Then we define 
\begin{align*}
\mathbf{V}(f_1,\ldots,f_s) = \set{(a_1,\ldots,a_n) \in \K^n : f_j(a_1,\ldots,a_n) = 0 \quad \forall j, 1 \leq j \leq s}.
\end{align*}
\end{definition}

Recall the notion of a variety associated to an ideal, which is the subset of $\K^n$ which simultaneously is a zero for all the polynomials in the ideal $I$. 
\begin{definition}
Given an ideal $I \subset \K[x_1,\ldots,x_n]$, we denote $\mathbf{V}(I)$ the set 
\begin{align*}
\mathbf{V}(I) = \set{(a_1,\ldots,a_n) \in \K^n : f(a_1,\ldots,a_n) = 0 \quad \forall f \in I}.
\end{align*}
This is an affine variety and in particular, $\mathbf{V}(\brak{f_1,\ldots,f_s}) = \mathbf{V}(f_1,...,f_s)$ (see e.g. [Proposition 9, page 81 \cite{Cox}]).  
\end{definition}

We are now ready to state a result on the non-solvability of a given system of polynomials which is equivalent to Hilbert's Nullstellensatz. 
\begin{theorem}[Weak Nullstellensatz (See Ch 1, Theorem 1 \cite{Cox})]
Let $I \subset \K[x_1,...,x_n]$ be an ideal over an algebraically closed field $\K$ satisfying $\mathbf{V}(I) = \emptyset$. Then $I = \K[x_1,...,x_n]$. 
\end{theorem}

The weak Nullstellensatz gives the following necessary and sufficient condition for the inconsistency of a given set of polynomial equations (over an algebraically closed field, e.g. $\mathbb C$). 
\begin{corollary}
Over an algebraically closed field, a given system of polynomial equations $f_1 = ... = f_s = 0$ does not have a solution if and only if $1 \in \brak{f_1,...,f_s}$. 
\end{corollary}

To computationally verify this, we need the concept of a Gr\"obner basis, which can in some way be considered an extension of a basis for the nullspace of a matrix in linear algebra. For this we need to fix a reasonable (total) ordering on monomials 
\[
x^\alpha = x_1^{\alpha_1} \ldots x_n^{\alpha_n},
\] 
where $\alpha \in \mathbb{Z}^n_{\geq 0}$ is a multi-index. A total ordering on monomials $\{x^\alpha\,:\, \alpha\in \Z^n_{\geq 0}\}$ is naturally equivalent to a total ordering on the set of multi indices $\mathbb{Z}^n_{\geq 0}$; see e.g. [Definition 1, page 55 \cite{Cox}]. There are many reasonable choices for orderings on $\Z^2_{\geq 0}$. A common choice of ordering is the lexicographic ordering where $\alpha > \beta$ if the leftmost non-zero entry in $\alpha-\beta$ is positive. However, many choices are possible and often preferable from a computational standpoint [Chapter 2, Section 2 of \cite{Cox}]. 

A monomial ordering allows one to define the {\em leading term} $\mathrm{LT}(f)$ of a polynomial $f$, defined to be largest term $cx^{\alpha}$ in that polynomial. In a similar manner, for an ideal $I$, we can define $\mathrm{LT}(I)$ the set of leading terms of all non-zero elements of $I$. With a monomial ordering on hand, we can now define the Gr\"obner basis of an ideal.
\begin{definition}[Gr\"obner Basis]
Fix a monomial order on $\K[x_1,...x_n]$. A finite subset $G = \set{g_1,...,g_k}$ of a non-zero ideal $I \subset \K[x_1,...x_n]$ is called a \emph{Gr\"obner basis} if 
\begin{align*}
\brak{\mathrm{LT}(I)} = \brak{\mathrm{LT}(g_1),...,\mathrm{LT}(g_k)}, 
\end{align*}
\end{definition}

One can show that every ideal $I \subset \K[x_1,\ldots x_n]$ has a Gr\"obner basis and moreover that any Gr\"obner basis $G = \set{g_1,\ldots,g_k}$ satisfies $\brak{g_1,\ldots,g_k} = I$ (see [Corollary 6; page 78 \cite{Cox}] ). 
We need to put an additional constraint in order to single out a unique, minimal Gr\"obner basis. 
\begin{definition}
Given a monomial ordering, a \emph{reduced Gr\"obner basis} for a polynomial ideal $I$ is a Gr\"obner basis $G$ of an ideal $I$ such that 
\begin{itemize}
	\item[(i)] The leading coefficient of $p=1$ for all $p \in G$.
	\item[(ii)] $\forall p \in G$ no monomial of $p$ lies in $\brak{\mathrm{G \setminus \set{p}}}$. 
\end{itemize}
\end{definition}

\begin{theorem}[Theorem 5, page 93 \cite{Cox}]
For a given monomial ordering and a given non-zero ideal $I \subset \K[x_1,...,x_n]$ there exists a unique reduced Gr\"obner basis. 
\end{theorem}

The following theorem summarizes now the sufficient condition we will use to prove inconsistency.
\begin{theorem}[Sufficient condition for inconsistency (See page 179, \cite{Cox})]\label{thm:suff-weak-null}
Let $\{f_1,\ldots f_s\}$ be a collection of polynomials in $\K[x_1,\ldots,x_n]$. If for a given monomial ordering, $\set{1}$ is the reduced Gr\"obner basis of the ideal $\langle f_1,...,f_s\rangle$, then $\mathbf{V}(f_1,\ldots,f_s)= \emptyset$ (i.e. $\{f_1,\ldots f_s\}$ are algebraically inconsistent) 
\end{theorem}
There exists many algorithms for computing the reduced Gr\"obner basis of a given ideal, we will use the Maple \cite{maple} implementation of the F4 algorithm \cite{F1999}. Computing a reduced Gr\"obner basis can be very computationally intensive hence it was important that we can take many steps to reduce the complexity of the system.

\subsection{Saturation of an ideal} \label{sec:Sat}
We are however, not quite finished. In our situation we do not quite have a simple system of polynomials but in fact we have a system of polynomials where solutions are constrained to stay away from a certain affine variety $\mathbf{V}(g) \in\K^n$ that is characterized as the zero set of a certain polynomial $g \in \K[x_1,\ldots,x_n]$. The constraint that solutions stay away from $\mathbf{V}(g)$ can then be characterized by the non vanishing of $g$. Our goal then is to understand when, for a given ideal $I\subseteq \K[x_1,\ldots,x_n]$, we have
\[
	\mathbf{V}(I) = \mathbf{V}(g),
\]
which implies that, away from the constraint set $\{g \neq 0\}$, the ideal $I$ has no common zeros. Algebraically, the idea here is to, in essence, mod out $g$ from the ideal $I$. This is commonly done via what is known as {\em saturation} of the ideal $I$ by $g$.
\begin{definition}[Saturation]
The saturation of an ideal $I$ of $\K[x_1,\ldots,x_n]$ by a polynomial $g$ is defined to be the set
\[
	I:g^\infty :=\{f\in \K[x_1,\ldots,x_n]\,:\, \text{there exists } m \geq 0 \text{ such that} fg^m \in I\}.
\]
\end{definition}

Our most important application of the saturation $I:g^\infty$ is the following version of the Strong Nullstellensatz.
\begin{theorem}[Strong Nullstellensatz (c.f. Ch 4, Theorem 10 \cite{Cox})] \label{thm:SNull}
Let $\K$ be an algebraically closed field and let $I$ be an ideal in $\K[x_1,\ldots,x_n]$, then for any polynomial $g\in \K[x_1,\ldots,x_n]$ we have
\[
	\mathbf{V}(I) = \mathbf{V}(g)
\]
if and only if the saturation $I:g^\infty = \K[x_1,\ldots,x_n]$. In other words, by the weak Nullstellensatz, 
\[
\mathbf{V}(I) = \mathbf{V}(g)\quad \text{if and only if}\quad \mathbf{V}(I:g^\infty) = \emptyset.
\]
\end{theorem}

In order to compute the saturation $I:g^\infty$, we make use of following trick that introduces a new variable $z$ that represents the inverse of $g$. This allows us to give a convenient computational condition for $\mathbf{V}(I) = \mathbf{V}(g)$.

\begin{theorem}[c.f. Chapter 4 Theorem 14 \cite{Cox}] \label{thm:Saturation}
Let $\{f_1,\ldots,f_s\}$ be a basis for an ideal $I\subseteq \K[x_1,\ldots, x_n]$, for an algebraically closed field $\K$. Let $z$ be a new variable and define the augmented ideal
\[
	\tilde{I} = \langle f_1,\ldots,f_s,zg-1\rangle \subseteq \K[x_1,\ldots,x_n,z],
\]
then 
\[
	\mathbf{V}(I:g^\infty) = \emptyset \quad \text{if and only if}\quad \mathbf{V}(\tilde{I}) = \emptyset.
\]
In particular, by the strong Nullstellensatz and Theorem \ref{thm:suff-weak-null}, if the reduced Gr\"{o}bner basis of $\tilde{I}$ is $\{1\}$, then
\[
	 \mathbf{V}(I) = \mathbf{V}(g).
\]
\end{theorem}\

\section{A key lemma regarding polynomials on the lattice} \label{sec:lattice}
This lemma is an important technical trick that is used repeatedly through out the paper to 
For example, we can deduce that if a polynomial vanishes at sufficiently many integer lattice points in some its variables, then the polynomials making up the coefficients of those variables must all vanish. 

\begin{lemma} \label{lem:Reducek}
	Let $P(x_1,...,x_n,y_1,...,y_m)$  be a polynomial $m+n$ complex variables and suppose that it is degree $J$ viewed as a polynomial in the last $m$ variables. 
	Suppose that $K \subset \mathbb Z^m$ contains 
	$\Z^m_{0,q}$ with $2q + 2 > J$ or $\Z^m_{0} \cap \set{\Z^{m}_{q} + i}$ with $2q + 1 > J$ and some $i \in \Z^m_0$.  
	Define the coefficients of $P$ viewed as a polynomial in the last $m$ variables 
	\begin{align*}
		\set{f_1,...,f_s} = \textup{coeffs}(P,\set{y_1,...,y_m}).  
	\end{align*}
	Furthermore for each $y = (y_1,\ldots,y_m)$ define the polynomial in $\C[x_1,\ldots,x_n]$ by
	\begin{align*}
	P_y(x_1,...,x_n) = P(x_1,...,x_n,y_1,...,y_m). 
	\end{align*}
	Then the following polynomial ideals are equivalent 
	\begin{align*}
	\brak{f_1,...,f_s} = \brak{P_y : y \in K}. 
	\end{align*}
\end{lemma}
\begin{proof}
	First, consider the case $m=1$.  
	The ideal generated by $P_y$ for $y \in K$ is given by the following
	\begin{align*}
	f(x) = \sum_{y \in K} h_y(x) P_y(x) = \sum_{y \in K} h_y(x) \sum_{0 \leq j \leq J} f_j(x) y^j, 
	\end{align*}
	and hence $\brak{P_y : y \in K} \subset \brak{f_1,...,f_s}$. 
	To see the other direction, consider any polynomial in $f \in \brak{f_1,...f_s}$
	\begin{align*}
	f(x) = \sum_{0 \leq j \leq J} h_j(x) f_j(x). 
	\end{align*}
Note that since there are at least $J+1$ distinct points in $K$ we can write each coefficient $f_j(x) = \sum_{y \in K} c_y P_y(x)$ for some rational coefficients $c_y$ (by inverting the Vandermonde matrix encoding the equations $P_y(x) = \sum_{0 \leq j \leq J} y^j f_j(x)$ $\forall y \in K$). Therefore, we can write 
	\begin{align*}
	f(x) = \sum_{0 \leq j \leq J} h_j(x) f_j(x) = \sum_{0 \leq j \leq J} \sum_{y \in K} h_j(x) c_y P_y(x), 	
	\end{align*}
	and therefore $\brak{f_1,...,f_s} \subseteq \brak{P_y : y \in K}$. 

Consider next the case $m > 1$. 
As in the $m=1$ case, it is straightforward to check that $\brak{P_y : y \in K} \subseteq \brak{f_1,...,f_s}$. 
The other inclusion can be proved analogously as well using the $m$-dimensional Vandermonde matrix. One may alternatively see it using an iterative argument as in the statement regarding affine varieties. 
Indeed, for any $(x_1,...,x_n,y_1,...,y_m)$ consider the set of points $y_{m+1}$ such that $y \in K$. 
Then we have the relationship 
\begin{align*}
P(x_1,...,x_n,y_1,...,y_{m+1}) = g_{J+1}(x_1,...,x_n,y_1,...,y_m) y^J_{m+1} + ... + g_1(x_1,...,x_n,y_1,...,y_m),
\end{align*}
for suitable coefficient polynomials $f_1,...f_{J+1}$. 
These coefficient polynomials can be solved for in terms of $P(x_1,...,x_n,y_1,...,y_{m+1})$ by a usual Vandermonde matrix for the points $y \in K$ %$\Z^{m+1}_{q}$. 
By the assumption on the set $K$, this argument can similarly be iterated until all of the $y_j's$ have been eliminated from the coefficients, proving the statement about the polynomial ideals. 
\end{proof}



\section*{Supplementary material}

## C.2  "Infinite Dimensional" Distinctness Proof

First we begin by loading the Groebner package in Maple

```
> restart:
> with(Groebner):
```

Next, we define some fuctions that define the various coefficients as functions:

The distorted lattice norm and skew product with distortion parameter $r$:

```
> lnorm := (i, r) ↦ i_1^2 + r^2 · i_2^2:
> skew := (i, k, r) ↦ r · (i_1 · k_2 − i_2 · k_1):
```

The Euler coefficient

```
> ccoef := (i, k, r) ↦ skew(i, k, r) · ( 1/lnorm(i, r) − 1/lnorm(k, r) ):
```

The algebraic expression for D

```
> Dcoef := (i, k, r) ↦ ccoef(i, k, r) · ccoef(i + k, k, r) − ccoef(i, k, r) · ccoef(i − k, k, r):
```

The saturation polynomial

```
> satpoly := (i, j, l, m, r) ↦ r · lnorm(i, r) · lnorm(j, r) · lnorm(l, r) · lnorm(m, r) · (lnorm(i + j, r)
       + lnorm(l + m, r)) · (lnorm(i + l, r) + lnorm(j + m, r)) · (lnorm(i + m, r) + lnorm(j + l, r)):
```

With these functions in hand, we now generate our set of polynomials associated to $D_i^k + D_j^k + D_l^k + D_m^k$: by calculating the numerator and then collecting the coefficients in k

```
> numerpoly := numer(Dcoef([i1, i2], [k1, k2], r) + Dcoef([j1, j2], [k1, k2], r) + Dcoef([l1, l2],
       [k1, k2], r) + Dcoef([m1, m2], [k1, k2], r)):
> degree(numerpoly, {k1, k2})
```
$$19 \tag{1}$$

```
> polys := coeffs(expand(numerpoly), {k1, k2}):
```

We define the variable ordering

```
> V := (z, m2, m1, l2, l1, j2, j1, i2, i1, r):
```

and then run the Groebner basis algorithm with "grevlex" ordering on the collection of polynomials with the extra $i + j + l + m$: condition and the saturating polynomial with the extra variable z

```
> polylist := {polys, i1 + j1 + l1 + m1, i2 + j2 + l2 + m2, z · satpoly([i1, i2], [j1, j2], [l1, l2], [m1, m2],
       r) − 1}:
> Basis(polylist, tdeg(V))
```
$$[1] \tag{2}$$

Since the Groebner basis is [1], this concludes the proof. QED

## C.3 Full Galerkin Truncated Distinctness Proof

First we begin by loading the Groebner package in Maple

> $restart$ :
> $with(Groebner)$ :

Next, we define some fuctions that define the various coefficients as functions:

The distorted lattice norm and skew product with distortion parameter $r$:

> $lnorm := (i, r) \mapsto i_1^2 + r^2 \cdot i_2^2$ :
> $skew := (i, k, r) \mapsto r \cdot (i_1 \cdot k_2 - i_2 \cdot k_1)$ :

The Euler coefficient

> $ccoef := (i, k, r) \mapsto skew(i, k, r) \cdot \left( \dfrac{1}{lnorm(i, r)} - \dfrac{1}{lnorm(k, r)} \right)$ :

The various algebrasic expressions that D can take

> $Dcoefp := (i, k, r) \mapsto ccoef(i, k, r) \cdot ccoef(i + k, k, r)$ :
> $Dcoefn := (i, k, r) \mapsto -ccoef(i, k, r) \cdot ccoef(i - k, k, r)$ :
> $Dcoef := (i, k, r) \mapsto ccoef(i, k, r) \cdot ccoef(i + k, k, r) - ccoef(i, k, r) \cdot ccoef(i - k, k, r)$ :

The saturation polynomial

> $g := (i, j, l, m, r) \mapsto r \cdot lnorm(i, r) \cdot lnorm(j, r) \cdot lnorm(l, r) \cdot lnorm(m, r) \cdot (lnorm(i + j, r) + lnorm(l + m, r)) \cdot (lnorm(i + l, r) + lnorm(j + m, r)) \cdot (lnorm(i + m, r) + lnorm(j + l, r))$ :

Now we generate all possible unordered lists of size 4, sampled with replacement from the list of 4 possible functions $[0, Dcoefp, Dcoefn, Dcoef]$ : removing the $[0, 0, 0, 0]$ : case.

> $cases := convert( (D \mapsto \{ seq( seq( seq( seq( sort( [D_i, D_j, D_l, D_m]), i = 1..4), j = 1..4), l = 1..4), m = 1..4) \}) ([0, Dcoefp, Dcoefn, Dcoef]) \setminus \{[0, 0, 0, 0]\}, list)$ :

By standard combinatorics (bars and stars counting), there should be

> $\dbinom{4 + 4 - 1}{4} - 1$

$$34 \qquad (1)$$

of these cases to consider. Indeed, we find that there are 34 total cases

> $numelems(cases)$

$$34 \qquad (2)$$

With all these functions in hand, we now generate our set of all possible polynomials associated to $D_i^k + D_j^k + D_l^k + D_m^k$ : and its various possible agebric forms by calculating the numerator. This generates a list of of polynomials in the variables $i1, i2, j1, j2, l1, l2, m1, m2, k1, k2$ : for each case defined above

> $numerpoly := map( case \mapsto expand( numer( case_1([i1, i2], [k1, k2], r) + case_2([j1, j2], [k1, k2], r) + case_3([l1, l2], [k1, k2], r) + case_4([m1, m2], [k1, k2], r) ) ), cases)$ :

Each polynomial is of degree at most 19 in k

> $map( poly \rightarrow degree(poly, \{k1, k2\}), numerpoly)$

$$[7, 5, 5, 11, 9, 9, 7, 7, 7, 15, 13, 13, 11, 11, 11, 9, 9, 9, 9, 19, 17, 17, 15, 15, 15, 13, 13, 13, 13, 11, 11, \qquad \textbf{(3)}$$
$$11, 11, 11]$$

We then collect the coefficients in $k$: , which gives a collection of polynomials in $i, j, l, m, r$: . We add to this collection the linear polynomial $i+j+l+m$: and the saturating polynomial augmented with the extra variable z. This generates a list of 34 collections of polynomials.

> $polys := map(p \mapsto \{coeffs(p, \{k1, k2\}), i1+j1+l1+m1, i2+j2+l2+m2, z \cdot g([i1, i2], [j1, j2], [l1, l2], [m1, m2], r) - 1\}, numerpoly)$ :

We now define the variable ordering

> $V := (z, m2, m1, l2, l1, j2, j1, i2, i1, r)$ :

and run the Groebner basis algorithm with "grevlex" ordering on each collection of polynomials. It will produce a list of Groebner bases for each case. Note that this step takes a little time to run (about 12 minutes total depending on CPU)

> $map(p \mapsto Basis(p, tdeg(V)), polys)$
$[[1], [1], [1], [1], [1], [1], [1], [1], [1], [1], [1], [1], [1], [1], [1], [1], [1], [1], [1], [1], \qquad \textbf{(4)}$
$[1], [1], [1], [1], [1], [1], [1], [1], [1], [1], [1], [1], [1], [1]]$

Since each Groebner basis is [1], this concludes the proof. QED

38



%% Bibliography
\addcontentsline{toc}{section}{References}
\bibliographystyle{abbrv}
\bibliography{bibliography}

\end{document}